\documentclass[9pt,reqno]{amsart}
\allowdisplaybreaks
\usepackage{amsmath,amstext,amssymb,epsfig}
\usepackage{tikz-cd}
\usepackage{graphicx}
\usepackage{mathrsfs}
\usepackage{xcolor}
\usepackage[a4paper, margin=1in]{geometry}  
\calclayout
\makeatletter
\renewcommand{\@seccntformat}[1]{\bf\csname the#1\endcsname.}
\renewcommand{\section}{\@startsection{section}{1}
	\z@{.7\linespacing\@plus\linespacing}{.5\linespacing}
	{\normalfont\upshape\bfseries\centering}}
\renewcommand{\@biblabel}[1]{\@ifnotempty{#1}{#1.}}
\makeatother
\theoremstyle{plain}
\newtheorem{thm}{Theorem}[section]
\newtheorem{lem}[thm]{Lemma}
\newtheorem{prop}[thm]{Proposition}

\theoremstyle{definition}
\newtheorem{ex}[thm]{Example}
\newtheorem{defn}[thm]{Definition}

\usepackage{tikz-cd}
\usepackage[parfill]{parskip}
\usepackage[german]{varioref}
\usepackage[all]{xy}
\usepackage{cancel}

\usepackage{tikz}
\usetikzlibrary{shapes.geometric, arrows, positioning}

\tikzstyle{startstop} = [rectangle, rounded corners, minimum width=3cm, minimum height=1cm,text centered, draw=black, fill=red!30]
\tikzstyle{process} = [rectangle, minimum width=3cm, minimum height=1cm, text centered, draw=black, fill=orange!30]
\tikzstyle{decision} = [diamond, minimum width=3cm, minimum height=1cm, text centered, draw=black, fill=green!30]
\tikzstyle{arrow} = [thick,->,>=stealth]
\usepackage{float}

\def\>{\succ}
\def\<{\prec}

\def\b{\beta}
\def\a{\alpha}
\def\l{\lambda}
\def\p{\partial}

\def\m{\mu}

\begin{document}
	\title[Sania Asif \textsuperscript{1}, Zhixiang Wu\textsuperscript{2}]{Cohomology and Homotopification of averaging operators on the Lie conformal algebras}
\author{Sania Asif\textsuperscript{1}, Zhixiang Wu\textsuperscript{2}}
\address{\textsuperscript{1} Institute of Mathematics, Henan Academy of Sciences, Zhengzhou, 450046, P.R. China.}
\address{\textsuperscript{2}Department of Mathematics, Zhejiang University, Hangzhou, Zhejiang Province, 310027, P.R. China} 
\email{\textsuperscript{1}11835037@zju.edu.cn}
\email{\textsuperscript{2}wzx@zju.edu.cn}
\date{\today}
\thanks{This work is funded by the Second batch of the Provincial project of the Henan Academy of Sciences (No. 241819105) and by NNSFC (No.12471038, No.12171129)}
\begin{abstract}
Building upon the work of Pavel in [P. Kolesnikov, Journal of Mathematical Physics, 56, 7 (2015)], we first present the cohomology of averaging operators on the Lie conformal algebras and use it to develop the cohomology of averaging Lie conformal algebras. We then introduce the homotopy version of averaging Lie conformal algebras and establish a connection between $2$-term averaging $\mathfrak{L}_\infty$-conformal algebra with the $3$-cocycle and crossed module of averaging Lie conformal algebra. Next, we study the non-abelian extension of the averaging Lie conformal algebras, showing that they are classified by the second non-abelian cohomology group. Finally, we demonstrate that a pair of automorphisms of averaging Lie conformal algebra is inducible if it can be seen as an image of a suitable Wells map. \end{abstract}\footnote{*The corresponding authors' emails: 11835037@zju.edu.cn ; wzx@zju.edu.cn.}
\maketitle \section{ Introduction}
 Conformal algebra is an interesting topic that was introduced by Kac in \cite{kac}, where the author provided the axiomatic description of operator product expansion of Chiral field theory or CFT rather using Fourier transforms. Conformal algebras later became a crucial tool for studying associative and Lie algebras, that satisfy the locality property. Finite Conformal algebras are useful for studying infinite dimensional algebra structures. A $2$-dimensional vertex operator algebra has connections to Lie conformal algebras, with the latter being a suitable framework for studying infinite-dimensional structures. The structure theory of Lie conformal algebras is developed and studied in \cite{HB, KP, L}. Later on, the study of operators and their cohomologies on conformal algebras are conducted in \cite{AWMB,AWW1,AWY,L,YL}. 
 \par Beyond conformal algebras, another key area of study involves averaging operators, which have a long history, originating in fluid dynamics through Reynolds' work on turbulence theory in 1895 \cite{R}. Later, averaging operators were explored in functional analysis \cite{K,bM,cM}, and gained popularity among algebraists after the work of Cao \cite{C} in 2000, primarily focusing on free commutative averaging algebra structures. Remember that an averaging operator on an Lie algebra $\mathfrak{L}$ is a linear mapping $P: \mathfrak{L}\to \mathfrak{L}$ that satisfies the following identity:
 \begin{align}
     [P(x),P(y)] = P([P(x),y]) = P([x,P(y)]), \quad for ~x,y \in \mathfrak{L}.
 \end{align} The study of averaging operators on Lie and Leibniz algebras, tensor hierarchies, and higher gauge theories is further generalized in \cite{BH,KoS,P,PG}. Lately, the cohomologies of averaging operators on associative and diassociative algebras are explored in \cite{D,WZ}. A Lie algebra along with averaging operator is called an averaging Lie algebra. The cohomology and deformation of averaging Lie algebras are studied using the derived bracket approach in \cite{MDH}.
In this paper we explore various aspects related to averaging operators, i.e., we studied cohomology, homotopy, non-abelian extension, and automorphisms of averaging Lie conformal algebras.
 \par Cohomology is a fundamental concept in algebra, offering insight into the ``holes" or obstructions within the algebraic structures. Consider an algebra $\mathfrak{L}$ and its module $\mathfrak{M}$, the $p$-cochain is given by $C^p(\mathfrak{L},\mathfrak{M})$, then the coboundary operator $d$ maps $C^p(\mathfrak{L},\mathfrak{M})$ to $C^{p+1}(\mathfrak{L},\mathfrak{M})$ and satisfies $d^2=0$. The kernel of $d$ denoted by $Z^p(\mathfrak{L},\mathfrak{M})$ consists of $p$-cocycles
and its range $B^{p+1}(\mathfrak{L},\mathfrak{M})$ comprises $p$-coboundaries, such that we have $B^{p}(\mathfrak{L},\mathfrak{M})\subseteq Z^p(\mathfrak{L},\mathfrak{M})$, then the quotient space $Z^p(\mathfrak{L},\mathfrak{M})/ B^{p}(\mathfrak{L},\mathfrak{M})$ is called the $pth$-cohomology group of $\mathfrak{L}$ with the coefficients in $\mathfrak{M}$, and is denoted by $H^p(\mathfrak{L},\mathfrak{M})$, see \cite{Cones}. Cohomology has a wide range of applications, extending beyond algebra into fields like topology and mathematical physics. There are various types of cohomology complexes. The Chevally Eilenberg cohomology complex is used to study cohomolgy of Lie algebras, the Hochschild cohomology complex is used to study cohomology of associative algebras. Describing cohomology, particularly in higher dimensional algebras can be a challenging task. 
The cohomology of various algebraic structures can be seen in \cite{B,FP,HM}. Lately, many researchers have explored the cohomology and deformation theory of various operator algebras and conformal algebras, see \cite{AWMB,AWW1,AWY,D2022, STZ}. In this paper we study the cohomology of averaging Lie conformal algebras, followed by cohomology of averaging operator on Lie conformal algebra. 
\par Higher algebraic structures can be categorized in two ways: those obtained through the homotopification of algebraic structures and those developed using category theory. In mathematics, the homotopy is used to classify geometric regions by exploring the various paths that can be drawn in a certain region. If two paths with common endpoints and can be continuously deformed into one another by leaving the end points fixed, and staying within their defined region is called homotopic. For the illustration of this fact, consider the Figure \ref{fig:1}, having two parts. In part $A$, the shaded region has a hole in it; $f$ and $g$ are homotopic paths, but $g'$ is not homotopic to $f$ or $g$ since $g'$ cannot be deformed into $f$ or $g$ without passing through the hole and leaving the region. There is no hole in part B, but all  paths have no common end points. Therefore all paths in B are not homotopic to each other.

Homotopy of an algebra was first appeared in the study of topological loop spaces by stasheff in \cite{S}. The homotopy algebra is also known as $\infty$-algebra in literature and plays a crucial role in the conformal field theory and gauge theory. The $\mathfrak{L}_\infty$-algebra is homotopification of Lie algebra, also called strongly homotopy Lie algebra. The $\mathfrak{L}_\infty$-algebra that is a Lie analog of $\mathfrak{A}_\infty$-algebra is introduced by Markel, Lada and stasheff in \cite{LM, LS}. The notion of $2$-term $\mathfrak{L}_\infty$-algebra was given in \cite{BC}, where it was discovered that, $2$-term $\mathfrak{L}_\infty$-algebra give rise to skeletal and strict Lie algebras. In simple words, $2$-term $\mathfrak{L}_\infty$-algebra consists of two terms with sequance of operation given by $l_k$, for $k\leq3$ and $k\in \mathbb{Z}$. The skeletal Lie algebra can be considered a ``simpler" version of the full structure, retaining essential properties while possibly losing some of the higher homotopical characteristics. On the other hand, in a strict Lie algebra, the higher operations $l_k$ for $k\geq3$ and $k\in \mathbb{Z}$ are trivial. This gives a conventional structure that satisfies the Jacobi identity and other classical properties of Lie algebras. Furthermore, \cite{BC} explores that the skeletal $\mathfrak{L}_\infty$-algebra is isomorphic to $3$-cocycle of the cohomology complex of Lie algebras and the strict $\mathfrak{L}_\infty$-algebra is isomorphic to the crossed module of Lie algebras. These results are also valid for $2$-term $\mathfrak{L}_\infty$-algebra. Later on in \cite{SD}, $\mathfrak{L}_\infty$-conformal algebras are proved homotopic to Lie conformal algebras. These results are further extended to Leibniz conformal algebras in \cite{DS}. In the present paper, we defined the $2$-term averaging $\mathfrak{L}_\infty$-conformal algebra and discussed skeletal and strict $2$-term $\mathfrak{L}_\infty $-conformal algebra in detail.
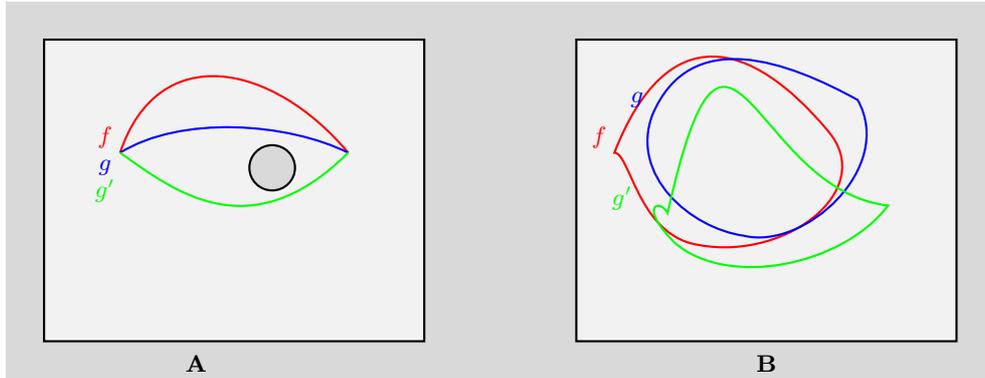
\begin{figure}[H]\begin{tikzpicture}
\fill[gray!30] (-6, -3) rectangle (7, 2);

\fill[gray!10] (-5.5, -2.5) rectangle (-0.5, 1.5); 
\fill[gray!30] (-2.5, -0.2) circle (0.3); 

\draw[thick, black] (-2.5, -0.2) circle (0.3); 

\draw[thick, black] (-5.5, -2.5) rectangle (-0.5, 1.5); 

\coordinate (common) at (-4.5, 0);
\coordinate (end) at (-1.5, 0); 

\draw[thick, red] (common) .. controls (-4, 1.5) and (-2.5, 1.2) .. (end);
\node at (-4.7, 0.2) {\textcolor{red}{$f$}};

\draw[thick, blue] (common) .. controls (-3.7, 0.5) and (-2.3, 0.4) .. (end);
\node at (-4.7, -0.2) {\textcolor{blue}{$g$}};

\draw[thick, green] (common) .. controls (-3.8, -0.5) and (-2.8, -1.3) .. (end);
\node at (-4.7, -0.5) {\textcolor{green}{$g'$}};

\node at (-3.5, -2.8) {\textbf{A}};

\fill[gray!10] (1.5, -2.5) rectangle (6.5, 1.5); 

\draw[thick, black] (1.5, -2.5) rectangle (6.5, 1.5); 

\draw[thick, red] (2, 0) .. controls (2.7, 1.8) and (3.8, 1.5) .. (4.8, 0.3) 
.. controls (5.5, -0.5) and (4.2, -1.5) .. (3, -1.2) 
.. controls (2.3, -1) and (2.2, 0) .. cycle;
\node at (1.8, 0.2) {\textcolor{red}{$f$}};

\draw[thick, blue] (2.5, 0.5) .. controls (3, 1.7) and (4.3, 1.2) .. (5.2, 0.7) 
.. controls (5.7, -0.2) and (4.5, -1.3) .. (3.7, -1.1) 
.. controls (3, -1) and (2.2, -0.3) .. cycle;
\node at (2.3, 0.7) {\textcolor{blue}{$g$}};

\draw[thick, green] (2.7, -0.8) .. controls (3.5, 2.8) and (3.8, -0.5) .. (5.6, -0.7) 
.. controls (5, -1.5) and (3.5, -1.8) .. (2.8, -1.2) 
.. controls (2.3, -0.7) and (2.6, -0.6) .. cycle;
\node at (2.1, -0.6) {\textcolor{green}{$g'$}};

\node at (4, -2.8) {\textbf{B}};

\end{tikzpicture}\caption{Illustration of homotopy in regions with and without a hole. In Part \textbf{A}, the shaded region contains a hole, and paths $f$ and $g$ are homotopic since they can be deformed into each other without leaving the region, whereas $g'$ is not homotopic to either $f$ or $g$ because it cannot be deformed without passing through the hole. In Part \textbf{B}, the region has no hole, but the paths do not share common endpoints, and thus, none of the paths are homotopic to one another.}
\label{fig:1}
\end{figure}

\par In this paper, we focuss on non-abelian extensions of averaging Lie conformal algebras, a topic that builds upon the general extension theory developed by Eilenberg and Maclane for abstract groups \cite{EM} and later generalized by Hochschild for Lie algebras\cite{HS}. Non-abelian extensions, being the most comprehensive type of extension, have seen recent developments in the context of Lie groups, Leibniz algebras, and Rota-Baxter algebras \cite{D,F,GZ,MDH,WZ}. Specifically, we define the second non-abelian cohomology group $H^2_{\text{nab}}(\mathfrak{L}_P, \mathfrak{H}_Q)$  for averaging Lie conformal algebras $\mathfrak{L}_P$ and $\mathfrak{H}_Q $. We further show that all equivalence classes of non-abelian extensions are classified by second non-abelian cohomology group $H^2_{\text{nab}}(\mathfrak{L}_P, \mathfrak{H}_Q)$. In Proposition 5.5, we show that choice of different sections preserves the equivalence relation between two $2$-cocycles.

\par Additionally, we address the inducibility problem of automorphisms in averaging Lie conformal algebras, extending the work of Wells, who introduced the concept in the context of abstract groups \cite{Wells} and further studied in \cite{PSY}. For a given non-abelian extension of averaging Lie conformal algebras $0 \to \mathfrak{H}_Q \to \mathfrak{E}_R \to \mathfrak{L}_P \to 0$, we construct an analogue of the Wells map $W: \text{Aut}(\mathfrak{H}_Q) \times \text{Aut}(\mathfrak{L}_P) \to H^2_{\text{nab}}(\mathfrak{L}_P, \mathfrak{H}_Q)$. Our main result, Theorem \ref{thm6.4}, shows that a pair of automorphisms $(\alpha,\beta) \in \text{Aut}(\mathfrak{H}_Q) \times \text{Aut}(\mathfrak{L}_P)$ is inducible if and only if the Wells map is trivial, indicating that the obstruction to inducibility lies in the second non-abelian cohomology group. Furthermore, in Theorem \ref{thm6.5} we construct the Wells exact sequence, which links various automorphism groups with the cohomology group $H^2_{\text{nab}}(\mathfrak{L}_P, \mathfrak{H}_Q)$.
 
 We can summarize the results of this paper by schematic diagram \ref{fig:2}, where the entity $A \to B$ means that either $A$ implies $B$ or $A$ is subcollection of $B$:
 \begin{figure}[H]
   \resizebox{\textwidth}{!}{ \begin{tikzpicture}[
    node distance=2cm,
    box/.style = {rectangle, draw, minimum height=1cm, minimum width=1.5cm, align=center},
    ->, thick, >=Stealth
    ]
    \node (A) [box] {Averaging $\mathfrak{L}_\infty$-conf. algebras};
    \node (B) [box, below left=of A] {$2$-term averaging \\ $\mathfrak{L}_\infty$-conf. algebras};
    \node (C) [box, below left=of B] {strict $2$-term averaging \\ $\mathfrak{L}_\infty$-conf. algebra};
    \node (D) [box, below=of B] {skeletal $2$-term averaging \\ $\mathfrak{L}_\infty$-conf. algebra};
    \node (E) [box, below=of C] {crossed modules};
    \node (F) [box, below right=of D] {$3$-cocycle};
    \node (H) [box, right=of B] {averaging \\Lie conf. algebras};
       \node (G) [box, below right=of H] {cohomology};
       \node (I) [box, below right=of A] {non-abelian\\ extension of \\
         $2$-term averaging \\ $\mathfrak{L}_\infty$ conf. algebra};
\node (J) [box, below right=of G] {2-cocycles};
\node (K) [box, right=of A] {automorphisms of \\ averaging Lie conf. algerba};
    \draw[->] (A) -- (B);
    \draw[->] (B) -- (C);
    \draw[->] (B) -- (D);
    \draw[<->] (C) -- (E);
    \draw[<->] (D) -- (F);
    \draw[->] (H) -- (G);
    \draw[<->] (B) -- (H); 
    \draw[->] (G) -- (F);
     \draw[->] (A) -- (I);
     \draw[->] (G) -- (J);
      \draw[<->] (I) -- (G);
       \draw[->] (A) -- (K);
        \draw[->] (K) -- (I);
\end{tikzpicture}}
    \caption{Schematic diagram.}
    \label{fig:2}
\end{figure}
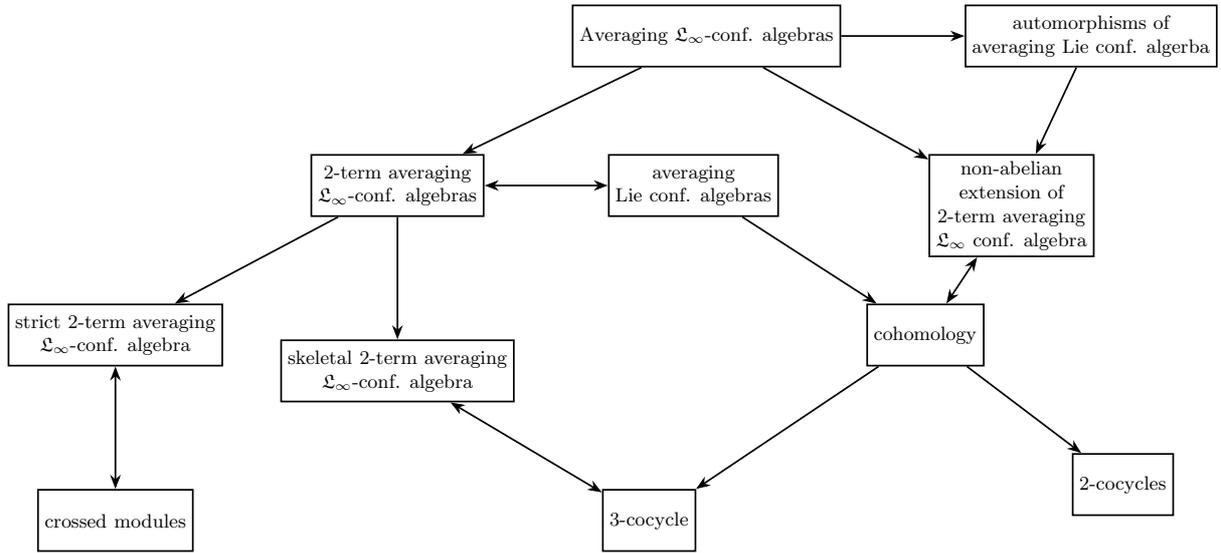
This paper is organized as follows:
In Section $2$, we provide preliminaries of averaging operator on Lie conformal algebras. In Section $3$, we describe the cohomology of averaging Lie conformal algebra. In Section $4$, we describe $2$-terms averaging $\mathfrak{L}_\infty $-conformal algebras and showed its connection with cohomology and crossed module of averaging Lie conformal algebras.  In Section 5, we discuss the non-abelian extension of averaging Lie conformal algebra and show that non-abelian extension can be classified by using second non-abelian cohomology group. Finally in Section $6$, we define automorphism of averaging Lie conformal algebra and discuss 
  the inducibility of the pair of automorphism of averaging Lie conformal algebras in terms of well defined Wells map.
  \section{Preliminaries of averaging operators on Lie conformal algebra}
 In this section, we define averaging Lie conformal algebra structure, followed by the introduction of averaging operator $P$ on Lie conformal algebras $\mathfrak{L}$. We then provide an averaging Lie conformal algebra structure on the direct sum of two Lie conformal algebras. Additionally, we define conformal representation on  $\mathfrak{L}$ and give the notion of relative averaging operator on $\mathfrak{L}$. Finally, we define the crossed module of Lie conformal algebra, which will be used in Section $4$.
 \begin{defn}
   Let $\mathfrak{L}$ be a $\mathbb{C}[\p]$-module, satisfying the axioms of Lie conformal algebra. An averaging operator on $\mathfrak{L}$ is a $\mathbb{C}[\p]$-linear map $P : \mathfrak{L}\to \mathfrak{L}$ that satisfies
$$P([P(x)_\l y]) = [P(x)_\l P(y)],$$
 for all $x,y \in \mathfrak{L}$ and $ \l \in  \mathbb{C}.$
 \end{defn} A Lie conformal algebra along with averaging operator is called averaging Lie conformal algebra, denoted by $\mathfrak{L}_P=(\mathfrak{L},[\cdot_\l \cdot], P)$.
 \begin{defn}
     Given two averaging Lie conformal algebra $\mathfrak{L}_P$ and $\mathfrak{L'}_{P'}$, a homomorphism of averaging Lie conformal algebras is a conformal algebra homomorphism $\psi : (\mathfrak{L},[\cdot_\l \cdot])\to (\mathfrak{L}',[\cdot_\l \cdot]')$ that satisfies $\psi P = P' \psi$.
 \end{defn}
 There is a lot of averaging Lie conformal algebras.
 \begin{ex}
     If $\mathfrak{L}$ is a Lie conformal algebra and direct sum of the $n$-copies of $\mathfrak{L}$ $(\textit{given by }\mathfrak{L}\oplus \mathfrak{L}\oplus \mathfrak{L}\oplus \cdots \oplus \mathfrak{L})$ defines a Lie conformal algebra with the $\l$-bracket :
\begin{align*}
    [(x_1, x_2,\cdots,x_n)_\l&(a_1,\cdots a_n)]\\&:=([{x_1}_\l {a_1}]_\mathfrak{L} , [{x_1}_\l {a_2}]_\mathfrak{L}- [{a_1}_{-\p-\l}x_2]_\mathfrak{L} ,\cdots, \overset{\text{i-th place for $i\geq 2$}}{[{x_1}_\l {a_i}]_\mathfrak{L}- [{a_1}_{-\p-\l}x_i]_\mathfrak{L}},\cdots [{x_1}_\l {a_n}]_\mathfrak{L}- [{a_1}_{-\p-\l}x_n]_\mathfrak{L}).
\end{align*} 
With the above Lie conformal bracket,  $\mathbb{C}[\p]$-module homomorphisms $P, P_i:\mathfrak{L} \oplus \mathfrak{L} \oplus \cdots \oplus \mathfrak{L}\to \mathfrak{L}\oplus \mathfrak{L}\oplus \cdots \oplus \mathfrak{L}$ defined by $P(x_1,x_2,\cdots,x_n)= (\sum_{i=2}^{n} x_i,0,\cdots,0)$ and $P_i(x_1,x_2,\cdots,x_n)= (x_i,0,\cdots,0)$, for $i\geq 2$ are averaging operator on the $nth$-direct sum Lie conformal algebras. \end{ex}  
 \begin{ex}
   The identity map $Id:\mathfrak{L} \to \mathfrak{L}$ is an averaging operator on any Lie conformal algebra $(\mathfrak{L},[\cdot_\l \cdot])$.  \end{ex}
   \begin{prop} Let $\mathfrak{L}$ be a Lie conformal algebra, and $P$ be an averaging operator. Then $\mathfrak{L}_P= (\mathfrak{L},[\cdot_\l\cdot], P)$ is an averaging Lie conformal algebra with a new  Lie bracket given by $[x_\l y]_P=[P(x)_\l y]$ and the same averaging operator $P$, denoted by $(\mathfrak{L}, [\cdot_\l\cdot]_P)$. Moreover, $P$ is a morphism of averaging Lie conformal algebras from  $(\mathfrak{L}, [\cdot_\l\cdot])$ to $ (\mathfrak{L}, [\cdot_\l\cdot]_P)$. 
 \end{prop}
An associative conformal algebra is called an averaging associative conformal algebra with an averaging operator $P$  if $P$ is $\mathbb{C}[\partial]$-linear endomorphism satisfying $P(P(a)_\lambda b)=P(a)_\lambda P(b)$ for all $a,b\in A$.
\begin{prop}
       Let $(\mathfrak{A},\cdot_\l)$ be an associative conformal algebra with averaging operator $P$. Then its commutator Lie conformal algebra, with $\lambda$-bracket given by
     \begin{align*}
         [x_\l y]= x\cdot_\l y- y\cdot_{-\p-\l} x, \quad for ~x,y\in \mathfrak{A}
     \end{align*}
     is an averaging Lie conformal algebra with the same averaging operator $P$.
   \end{prop}
   Now, let $\mathfrak{L}$ be a Lie conformal algebra and $\mathfrak{M}$ be a $\mathbb{C}[\p]$-module, a $\mathbb{C}$-linear map $\rho: L \to Cend(\mathfrak{M})$ is called a conformal representation of $\mathfrak{L}$, if it satisfies the following equations, for $x, y\in \mathfrak{L}$:
\begin{align}
\rho \partial=& \partial \rho, \\ \rho(\partial(x))_\lambda =& -\lambda \rho(x)_\lambda,\\ \rho([x_{\lambda}y])_{\lambda+\mu} =& \rho(x)_{\lambda}\rho(y)_\mu - \rho(y)_{\mu}\rho(x)_{\lambda}.
\end{align}
We denote the conformal representation by the pair $(\mathfrak{M},\rho)$. If $\mathfrak{M}=\mathfrak{L}$, then $\rho$ is an adjoint representation of $\mathfrak{L}$, defined by $ad(x)_\l y= [x_\l y]$, for all $x, y\in \mathfrak{L}$.
\begin{prop}
Let $(\mathfrak{L}, [\cdot_\l\cdot])$ be a Lie conformal algebra and $(\mathfrak{M}, \rho)$ be its conformal representation, then we can define the relative averaging operator on $\mathfrak{L}$ w.r.t $\mathfrak{M}$ by a $\mathbb{C}$-linear map $T: \mathfrak{M} \to \mathfrak{L}[\l]$ that satisfies
\begin{align}
[{T(m)}_\l T(n)]= T({\rho(T m)}_\l n),\quad \textit{for all } m,n \in \mathfrak{M}.
\end{align}Note that the relative averaging operator is also termed an embedding tensor.
\end{prop}
Furthermore, we can also define the $\l$-bracket on the semi-direct sum of Lie conformal algebra $\mathfrak{L}$ and conformal module $\mathfrak{M}$ by \begin{align}\label{eqrep2}
    {[{(x+m)}_\l (y+n)]}_\ltimes=([x_\l y], \rho(x)_\l m-\rho(y)_{-\p-\l}n),\quad \forall (x, m),(y,n)\in \mathfrak{L}\oplus \mathfrak{M}.
\end{align}In the next proposition, we show that there is a close relationship between the relative averaging operator and averaging operators on a semi-direct sum Lie conformal algebra.
\begin{prop}An operator $T: \mathfrak{M}\to \mathfrak{L}[\l]$ is a relative averaging operator on the Lie conformal algebra $\mathfrak{L}$, iff $P_T$ is an averaging operator on the semi-direct sum Lie conformal algebra $\mathfrak{L}\oplus \mathfrak{M}$, defined by $P_T(x, m)= (T(m), 0)$. \end{prop}
\begin{defn}
A representation of an averaging Lie conformal algebra $\mathfrak{L}_P$ consists of the triple $(\mathfrak{M},\rho,\phi)$, where $(\mathfrak{M},\rho)$ is a  conformal representation of the Lie conformal algebra and $\phi\in Cend(\mathfrak{M})$ is a $\mathbb{C}[\partial]$-linear map, such that following identity holds: 
\begin{align}\label{eqrep}
\rho(Px)_\l \phi(m)=\phi(\rho(Px)_\l m)=\phi(\rho(x)_\l \phi(m)),
\end{align} for all $x\in \mathfrak{L}$, $m \in \mathfrak{M}$.
We denote the representation of $\mathfrak{L}_P$ by $\mathfrak{M}_\phi$. 
\end{defn}
An averaging Lie conformal algebra $\mathfrak{L}_P$ is  a representation of itself with $\phi=P$.
 Further, we can define the averaging operator structure on $\mathfrak{L}_P\oplus \mathfrak{M}_\phi$.
   \begin{prop}\label{propnew} Either $P_1(x+m)=x$, or $P_2(x+m)=\phi(m)$, or $P_3(x+m)=P(x)+\phi(m)$ is an averaging operator of the semi-direct sum Lie conformal algebra $\mathfrak{L}_P\oplus \mathfrak{M}_\phi$.\end{prop}
\begin{proof}
    The proof is straight forward and can be followed by assuming Eqs. \eqref{eqrep2} and \eqref{eqrep}.
    \end{proof}
  \begin{ex}
     If $\mathfrak{L}_P$ is an averaging Lie conformal algebra and $\mathfrak{M}$ is a tensor product of the $n$-copies of $\mathfrak{L}_P$. Then $\mathfrak{M}_\phi$ becomes a representation of $\mathfrak{L}_P$ with 
\begin{align*}
    x_\lambda(a_1\otimes \cdots\otimes  a_n):=
    \sum\limits_{i=1}^n( {a_1}\otimes \cdots a_{i-1}\otimes x_\lambda a_i\otimes a_{i+1}\otimes\cdots\otimes a_n)
\end{align*} 
and $$\phi(a_1\otimes \cdots\otimes  a_n):= \sum\limits_{i=1}^n {a_1}\otimes \cdots a_{i-1}\otimes P( a_i)\otimes a_{i+1}\otimes\cdots\otimes a_n.$$
Then one can obtain averaging Lie conformal algebras  $\mathfrak{L}_P\oplus \mathfrak{M}_\phi$. \end{ex}  
\begin{defn}A differential crossed module of a Lie conformal algebra consists of a pair of Lie conformal algebras $\mathfrak{L}_0$ and  $\mathfrak{L}_1$, equipped with a Lie conformal algebra homomorphism: 
\begin{align*}
    d:& \mathfrak{L}_1 \to \mathfrak{L}_0
\end{align*} and a conformal $\mathbb{C}$-linear map  
$\rho: \mathfrak{L}_0\otimes \mathfrak{L}_1 \to  \mathfrak{L}_1[\l] $ defined by $(x,v) \mapsto \rho(x)_\l(v)$ that makes $\mathfrak{L}_1$ as a conformal representation of $\mathfrak{L}_0$, satisfying the following properties 
\begin{align*}
d(\rho(x)_\l(v)) &= [x_\l dv]_{\mathfrak{L}_0}\\
\rho(d(u))_\l (v) &= [u_\l v]_{\mathfrak{L}_1}
\end{align*}for all $x \in \mathfrak{L}_0$ and  $u,v \in \mathfrak{L}_1$.
A differential crossed module is equivalent to a dg-Lie conformal algebra structure on a chain complex $(\mathfrak{L}_1 \xrightarrow{d} \mathfrak{L}_0)$ concentrated in degrees $0$ and $1$. The components of the dg-Lie conformal bracket are:
\begin{itemize}
    \item the Lie conformal bracket $[\cdot_\l \cdot] : \mathfrak{L}_0 \otimes \mathfrak{L}_0 \to \mathfrak{L}_0 [\l]$,
    \item the mixed Lie conformal bracket $[\cdot_\l \cdot] : \mathfrak{L}_0 \otimes \mathfrak{L}_1 \to \mathfrak{L}_1[\l]$, which corresponds to the action: $[x_\l v] = \rho(x)_\l(v)$, for all $x\in  \mathfrak{L}_0$ and $v\in \mathfrak{L}_1.$
\end{itemize}
\end{defn}
The notion of differential crossed module provides a way to characterize the structure of a strict Lie $2$- conformal algebra in terms of two ordinary Lie conformal algebras.
\section{Cohomology of averaging Lie conformal algebra}
To express the cohomology of the averaging Lie conformal algebra, we begin by defining the cohomology of averaging operator on a Lie conformal algebra and provide some recap of the cohomology of a Lie conformal algebra. This section is divided into three subsections; the first subsection is devoted to studying the cohomology complex for  Lie conformal algebras. In the second subsection, we define the cohomology complex for averaging operators on the Lie conformal algebras. Finally, we provide the cohomology complex of averaging Lie conformal algebras.
\subsection{ The cohomology complex for a  Lie conformal algebra with representation:}	
 Let $(\mathfrak{L}, [\cdot_\l\cdot] )$ be a  Lie conformal algebra with a representation $(\rho, \mathfrak{M} )$ on a $\mathbb{C}[\p]$-module $\mathfrak{M}.$ Assume that $C^p(\mathfrak{L}, \mathfrak{M})$ is a space of $p$-cochains that is equipped with the $\mathbb{C}$-linear maps $$f_{\l_1, \l_2, \cdots, \l_{p-1}}: \mathfrak{L}^{\otimes p}\to \mathbb{C}[\l_1,\l_2,\cdots,\l_{p-1}]\otimes \mathfrak{M},$$ defined by $$x_1\otimes x_2\otimes\cdots \otimes x_p \mapsto f_{\l_1,\l_2,\cdots,\l_{p-1}}(x_1,x_2,\cdots,x_p).$$ These maps satisfy the following conditions of\\ \begin{enumerate}
\item Conformal sesqui-linearity:
	\begin{equation}\label{eq9}
	\begin{aligned}
	f_{\l_1, \l_2,\cdots, \l_{p-1}}&(x_1,x_2,\cdots,\p(x_i),\cdots,x_p)\\&=\begin{cases}-\l_{i} f_{\l_1,\l_2,\cdots,\l_{p-1}}(x_1, x_2, \cdots , x_p), &i= 1, \cdots, p-1, \\
	 (\p+\l_1+\l_2+\cdots+\l_{p-1}) f_{\l_1,\l_2,\cdots,\l_{p-1}}(x_1, x_2, \cdots , x_p),&i=p. \end{cases}
	\end{aligned}\end{equation}
	and 
     \item Conformal skew-symmetry:
	\begin{equation}\label{eq11}
	\begin{aligned}
	f_{\l_1,\l_2,\cdots,\l_{p-1}}(x_1,x_2,\cdots,x_p)=(-1)^\tau f_{\l_{\tau(1)},\l_{\tau(2)},\cdots,\l_{\tau(p-1)}}(x_{\tau(1)}, x_{\tau(2)}, \cdots , x_{\tau(p)})|_{\l_{p}\mapsto \l_{p}^\dagger},
	\end{aligned}\end{equation}
	where $\l_{p}^\dagger=-\sum_{i=1}^{p-1}\l_{i}-\p^{\mathfrak{M}}$, for all $x_1, x_2, \cdots, x_p\in \mathfrak{L}$. Moreover, the above equation shows that the $ \mathbb{C}$-linear map $f$ is skew-symmetric concerning simultaneous permutation of $ x_{j} $'s and $\l_j$'s.
 \end{enumerate}
The linear coboundary map $\delta : C^p (\mathfrak{L}, \mathfrak{M} ) \to C^{p+1} (\mathfrak{L},\mathfrak{M})$ is given by : 
\begin{equation}\label{coboundarymap}\begin{aligned}(\delta f)_{\l_1,\cdots,\l_{p}}(x_1,\cdots, x_{p+1})=& \sum_{i=1}^{p+1} (-1)^{i+1}\rho((x_i))_{\l_i}f_{\l_1,\cdots,\hat{\l_{i}},\cdots, \l_{p}}(x_1,\cdots,\hat{x_i},\cdots, x_{p+1})\\&+ \sum_{i<j}(-1)^{i+j} f_{\l_i+ \l_j ,\l_1,\cdots, \hat{\l_{i}}, \cdots, \hat{\l_{j}}, \cdots, \l_{p}}([{x_{i}}_{\l_i} x_j], x_1, \cdots, \hat{x_i},\cdots, \hat{x_j},\cdots, x_{p+1}), \end{aligned}\end{equation} 
for $x_1, x_2, \cdots, x_{p+1} \in \mathfrak{L}$, where  $\hat{x_i}$ shows the omission of the entry $x_{i}$  (see \cite{AWMB, ZYC} for more detail).\\
 Consider the space $C^*(\mathfrak{L},\mathfrak{M}) :=\bigoplus_{p\geq1}C^p(\mathfrak{L}, \mathfrak{M})$ is the graded space of $p$-cochains. Assuming that $f\in C^p(\mathfrak{L}, \mathfrak{M})$, $\delta (f)\in C^{p+1}(\mathfrak{L},\mathfrak{M})$, and $ \delta^2 (f)=0$, implies that $(C^* (\mathfrak{L}, \mathfrak{M}), \delta )$ is a cochain complex for the  Lie conformal algebra $(\mathfrak{L}, [\cdot_\l\cdot] )$ with coefficients in the representation $(\rho,\mathfrak{M})$. We denote the cohomology space associated to the cochain complex $(C^* (\mathfrak{L}, \mathfrak{M}), \delta )$ by $\mathcal H^*(\mathfrak{L}, \mathfrak{M})$.
\subsubsection{ The cohomology complex for a Lie conformal algebra in terms of Maure-Cartan element:} Observe that, if $\mathfrak{M} = \mathfrak{L}$ and the action $\rho: \mathfrak{L}\otimes \mathfrak{M} \to \mathfrak{M}[\l]$ is given by the underlying  Lie conformal bracket. In this case, the coboundary map in Eq. \eqref{coboundarymap} is given be \begin{equation}\label{coboundarymap2}\begin{aligned}(\delta  f)_{\l_1,\cdots,\l_{p}}&(x_1,\cdots, x_{p+1})\\=& \sum_{i=1}^{p+1} (-1)^{i+1} [  {x_i}  _{\l_i}f_{\l_1,\cdots,\hat{\l_{i}},\cdots, \l_{p}}(x_1,\cdots,\hat{x_i},\cdots, x_{p+1})]\\&+ \sum_{i<j}(-1)^{i+j} f_{\l_i+ \l_j ,\l_1,\cdots, \hat{\l_{i}}, \cdots, \hat{\l_{j}}, \cdots, \l_{p}}([{x_{i}}_{\l_i} x_j],  x_1 , \cdots, \hat{x_i},\cdots, \hat{x_j},\cdots, x_{p+1}). \end{aligned}\end{equation}The cochain complex thus obtained is  denoted by  $(C^* (\mathfrak{L},\mathfrak{L}),\delta ).$ And this complex is the same as the cochain complex of the  Lie conformal algebra $(\mathfrak{L}, [\cdot_\l\cdot])$ (defined in \cite{SK}). The cohomology of the cochain complex $(C^* (\mathfrak{L}, \mathfrak{L}), \delta )$ is denoted by $\mathcal H^* (\mathfrak{L}, \mathfrak{L})$.
 \par Let us now consider the graded vector space $C^*(\mathfrak{L},\mathfrak{L})=\bigoplus_{p\in \mathbb{Z}}C^p (\mathfrak{L},\mathfrak{L})$, where $C^p(\mathfrak{L},\mathfrak{L})$ is the space of $\mathbb{C}$-linear maps satisfying Eqs. \eqref{eq9}-\eqref{eq11}. From \cite{AWY}, we know that there is  a circle product $f\circledcirc g: C^{p} (\mathfrak{L}, \mathfrak{L})\otimes C^{q} (\mathfrak{L}, \mathfrak{L})\to C^{p+q-1} (\mathfrak{L}, \mathfrak{L})$  given by 
\begin{equation}\label{circleproduct}
		\begin{aligned}
		(f\circledcirc g)&_{\l_1,\cdots,\l_{p+q-2}}(x_1,\cdots,x_{p+q-1})=\sum_{\tau\in S_{q,p-1}} (-1)^{|\tau|} f_{\l_{\tau(1)}+\cdots+\l_{\tau({q})},\l_{\tau({q+1})},\cdots,\l_{\tau({p+q-2})}}\\& (g_{\l_{\tau(1)},\cdots,\l_{\tau({q-1})}}(x_{\tau(1)},\cdots, x_{\tau(q)}), x_{\tau(q+1)},\cdots,  x_{\tau(p+q-1)}).
		\end{aligned}
		\end{equation} 
		for $x_1,\cdots ,x_{p+q-1}\in \mathfrak{L}$, where $S_{q,p-1}$ denotes $\{1,2,\cdots,p+q-1\}$-shuffle in the permutation group of $S_{q+p-1}$ and $|\tau|$ denotes the sign of the permutation $\tau$.
	 	Now by using Eq. \eqref{circleproduct}, we can define a Nijenhuis-Richardson bracket (also called as $NR$-bracket) of degree $-1$ on graded vector space $C^* (\mathfrak{L}, \mathfrak{L})$ as follows:\begin{align}
	 	    [f,g]_{NR} :=f \circledcirc g -(-1)^{(p-1)(q-1)}g\circledcirc f.
	 	\end{align}
  The graded vector space $ C^{*}(\mathfrak{L}, \mathfrak{L}) $ along with the Nijenhuis-Richardson bracket $[\cdot, \cdot]_{NR}$ forms a differential graded Lie algebra $(C ^{*}(\mathfrak{L}, \mathfrak{L}), [\cdot, \cdot]_{NR})$. This differential graded Lie algebra, abbreviated as $deLa$, 
	is also called controlling Lie algebra. The differential operator on the graded Lie algebra $(C^{*}(\mathfrak{L}, \mathfrak{L}),[\cdot,\cdot]_{NR})$ can be defined in terms of Maure-Cartan element. Maurer-Cartan element, denoted by $\eta_c$. It is basically an element from the space of $2$-cochain $C^2(\mathfrak{L}, \mathfrak{L})$ satisfying the following condition:
  $[\eta_c,\eta_c]_{NR}=0$. Furthermore, for any $g\in C^q(\mathfrak{L}, \mathfrak{L})$ the differential of the  $dgLA$ is given by $d_{\eta_c}:=[\eta_c,-]_{NR}$, 
for $\eta_c\in C^2(\mathfrak{L},\mathfrak{L})$. 
\begin{align*}
	d_{\eta_c}(g)=& [\eta_c,g]_{NR}\\=& {(\eta_c\circledcirc g -(-1)^{(1)(q-1)}g\circledcirc \eta_c)} _{\l_1,\cdots,\l_q}(x_1,\cdots,x_{q+1})
	\\=& (\eta_c \circledcirc g)_{\l_1,\l_2,\cdots,\l_q}  (x_1,\cdots,x_{q+1})+(-1)^{q}(g\circledcirc \eta_c)_{\l_1,\l_2,\cdots,\l_q} (x_1,\cdots,x_{q+1})
	\\=& \sum_{\tau\in S_{q,1}} (-1)^{|\tau|} {\eta_c}_{\l_{\tau(1)}+\cdots+\l_{\tau({q})}} (g_{\l_{\tau(1)},\cdots,\l_{\tau({q-1})}}(x_{\tau(1)},\cdots, x_{\tau(q)}),  x_{\tau(q+1)}) \\&
	+(-1)^{q} \sum_{\tau\in S_{2,q-1 }} (-1)^{|\tau|} g_{\l_{\tau(1)}+\l_{\tau({2})},\l_{\tau({3})},\cdots,\l_{\tau({q})}} ({\eta_c}_{\l_{\tau(1)}}(x_{\tau(1)}, x_{\tau(2)}), x_{\tau(3)},\cdots,  x_{\tau(q+1)}) 
	\\=& \sum_{\tau\in S_{q,1}} (-1)^{|\tau|}  [{g_{\l_{\tau(1)},\cdots,\l_{\tau({q-1})}}(x_{\tau(1)},\cdots, x_{\tau(q)})}_{\l_{\tau(1)}+\cdots+\l_{\tau({q})}}  x_{\tau(q+1)}]\\&
+(-1)^{q} \sum_{\tau\in S_{2,q-1}} (-1)^{|\tau|} g_{\l_{\tau(1)}+\l_{\tau({2})},\l_{\tau({3})},\cdots,\l_{\tau({q})}} ([{x_{\tau(1)}}_{\l_{\tau(1)}} x_{\tau(2)}], x_{\tau(3)}, \cdots, x_{\tau(q+1)}).
\end{align*}
 Since it satisfies $d_{\eta_c}^2=0$, the triple $(C^{*}(\mathfrak{L}, \mathfrak{L}),[\cdot,\cdot]_{NR}, d_{\eta_c})$ forms a differential graded Lie algebra. %
 \subsection{ The cohomology complex for averaging operators on a  Lie conformal algebra with representation:}
 Let $\mathfrak{L}_P=(\mathfrak{L},[\cdot_\l \cdot],P)$ be an averaging Lie conformal algebra with representation $(\mathfrak{M}_\phi,\rho_P)$. The cochain of $\mathfrak{L}_P$ with the coefficients in $(\mathfrak{M}_\phi,\rho_P)$ is given by 
 $C^{*}_{AO}(\mathfrak{L}, \mathfrak{M})=\oplus_{p\geq 1}C^{p}_{AO}(\mathfrak{L}, \mathfrak{M})$, where 
 $$C^{p}_{AO}(\mathfrak{L},\mathfrak{M}):=\{f\in  Hom(\mathfrak{L}^{\otimes p},\mathfrak{M})|f\text{is conformal skew-symmetric}\}.$$
 We can define the coboundary operator on this cochain by $\partial_{AO}:C^{p}_{AO}(\mathfrak{L}, \mathfrak{M})\to C^{p+1}_{AO}(\mathfrak{L}, \mathfrak{M})$, given by:\begin{equation}\begin{aligned}&(\partial_{AO} f)_{\l_1,\cdots,\l_{p}}(x_1,\cdots, x_{p+1})\\=& \sum_{i=1}^{p+1} (-1)^{i+1}{\rho_P(x_i)}_{\l_i}f_{\l_1,\cdots,\hat{\l_{i}},\cdots, \l_{p}}(x_1,\cdots,\hat{x_i},\cdots, x_{p+1})\\&
+ \sum_{1\leq i<j\leq p+1}(-1)^{i+j} f_{\l_i+ \l_j ,\l_1,\cdots, \hat{\l_{i}},\cdots, \hat{\l_{j}},\cdots, \l_{p}}({[{x_{i}}_{\l_i} x_j]}_P,  x_1 , \cdots, \hat{ x_i},\cdots, \hat{x_j},\cdots,x_{p+1})
\\=& \sum_{i=1}^{p+1} (-1)^{i+1}{\rho( Px_i)}_{\l_i}f_{\l_1,\cdots,\hat{\l_{i}},\cdots, \l_{p}}(x_1,\cdots,\hat{x_i},\cdots, x_{p+1})\\&
+ \sum_{1\leq i<j\leq p+1}(-1)^{i+j} f_{\l_i+ \l_j ,\l_1,\cdots, \hat{\l_{i}},\cdots, \hat{\l_{j}},\cdots, \l_{p}}({[{P(x_{i})}_{\l_i} x_j]}, x_1, \cdots, \hat{x_i},\cdots, \hat{x_j},\cdots,(x_{p+1})),
  \end{aligned}\end{equation} where $f\in C^{p}_{AO}(\mathfrak{L}, \mathfrak{M})$ and $x_1, x_2,\cdots, x_{n+1}\in \mathfrak{L}$.
It can be check that $\partial_{AO}^p\circ \partial_{AO}^{p+1}=0$, that is, it is indeed a coboundary operator for the induced Lie conformal algebra $(\mathfrak{L}, [\cdot_\l \cdot]_P)$. Therefore, $ (C^{*}_{AO}(\mathfrak{L}, \mathfrak{M}),\partial_{AO})$ is a cochain complex of the averaging Lie conformal algebra $\mathfrak{L}_P$, with the coefficients in the representation $(\mathfrak{M}_\phi, \rho_P)$, Cohomology of this cochain complex is denoted by $H^{*}_{AO}(\mathfrak{L}_P, \mathfrak{M}_\phi )$. If $(\mathfrak{M}_\phi, \rho_P)=(\mathfrak{L}, [\cdot_\l \cdot]_P)$, then corresponding cohomology is denoted by $ H^{*}_{AO}(\mathfrak{L})$.

To compare the complex $C^*_{AO}(\mathfrak{L},\mathfrak{M}),\partial_{AO})$ with the complex $C^*(C(\mathfrak{L},\mathfrak{M}),\partial),$ we define a chain map as follow.

\begin{defn}Let $(\mathfrak{M}_\phi, \rho)$ be a representation of an averaging Lie conformal algebra $(\mathfrak{L}, [\cdot_\l \cdot]_P)$. Define a map $\xi^p: C^{p}(\mathfrak{L}, \mathfrak{M} )\to C^{p}_{AO}(\mathfrak{L}, \mathfrak{M} )$ by 
    \begin{align}
        \xi^pf_{\l_1,\l_2,\cdots,\l_{p-1}}(x_1,x_2,\cdots,x_p) = f_{\l_1,\l_2,\cdots,\l_{p-1}}(Px_1,Px_2,\cdots,Px_p)- (\phi \circ f)_{\l_1,\l_2,\cdots,\l_{p-1}}(Px_1,x_2,\cdots,x_p),
    \end{align} for all $x_1,\cdots,x_p\in L$ and $f\in  C^{p}(\mathfrak{L}, \mathfrak{M} ).$ 
\end{defn}
\begin{lem}\label{lemma3.2}For any $x_1,\cdots,x_{p+1}\in \mathfrak{L}$ and $f\in  C^{p}(\mathfrak{L}, \mathfrak{M} )$, we have
   \begin{align}
        \xi^{p+1}(\delta^p(f))_{\l_1,\l_2,\cdots,\l_{p}}(x_1,x_2,\cdots,x_{p+1}) = \partial^p
(\xi^{p}(f))_{\l_1,\l_2,\cdots,\l_{p}}(x_1,x_2,\cdots,x_{p+1}).
   \end{align} 
\end{lem}
\begin{proof}
    First, consider that
\begin{align*}
    \xi^{p+1}(\delta^p(f))_{\l_1,\l_2,\cdots,\l_{p}}&(x_1, x_2, x_3, \ldots , x_{p+1}) \\=&
(\delta^pf)_{\l_1,\l_2,\cdots,\l_{p}}(P x_1, P x_2, \ldots, P x_p, P x_{p+1}) \\&- (\phi \circ (\delta^pf))_{\l_1,\l_2,\cdots,\l_{p}}(Px_1,  x_2, \ldots,  x_p,  x_{p+1}) \\=& \sum_{i=1}^{p+1}(-1)^{i+1} [(P x_i)_{\l_i}f_{\l_2,\cdots,\hat{\l_{i}},\cdots, \l_{p}}(P x_1,\cdots\hat{P x_i},\cdots, P x_{p+1})]\\& 
	 + \sum_{1\leq i<j\leq p+1}(-1)^{i+j} f_{\l_i+ \l_j ,\l_1,\cdots, \hat{\l_{i}},\cdots, \hat{\l_{j}},\cdots, \l_{p}}([{P x_{i}}_{\l_i} P x_j], P x_1, \cdots, \hat{P x_i},\cdots, \hat{P x_j},\cdots,P x_{p+1})
   \\&- \phi([(P x_1)_{\l_1}f_{\l_2,\cdots,\hat{\l_{1}},\cdots, \l_{p}}( x_2,\cdots,x_{p+1})]\\& 
	 + \sum_{1<j\leq p+1}(-1)^{1+j} f_{\l_1+ \l_j ,\l_2,\cdots,  \hat{\l_{j}},\cdots, \l_{p}}([{P x_{1}}_{\l_1}  x_j], x_2,\cdots, \hat{x_j},\cdots,x_{p+1}))
   \\&- \phi(\sum_{i=2}^{p+1}(-1)^{i+1} [ {x_i}_{\l_i}f_{\l_2,\cdots,\hat{\l_{i}},\cdots, \l_{p+1}}(P x_1,x_2,\cdots\hat{x_i},\cdots,x_{p+1})]\\& 
	 + \sum_{2\leq i<j\leq p+1}(-1)^{i+j} f_{\l_i+ \l_j ,\l_1,\cdots, \hat{\l_{1}},\cdots, \hat{\l_{j}},\cdots, \l_{p}}([{x_{i}}_{\l_i}  x_j], P x_1,x_2, \cdots, \hat{x_i},\cdots, \hat{x_j},\cdots,x_{p+1})).
\end{align*} Next
\begin{align*}
\partial_{AO}^p(\xi^p(f))_{\l_1,\cdots, \l_{p}}&(x_1, x_2, x_3, \ldots , x_{p+1})\\
	=&\sum_{i=1}^{p+1}(-1)^{i+1} \rho(P(x_i))_{\l_i}(\xi^pf)_{\l_1,\cdots,\hat{\l_{i}},\cdots, \l_{p}}(x_1,\cdots,\hat{x_i},\cdots, x_{p+1})\\&
 + \sum_{1\leq i<j\leq p+1}(-1)^{i+j} (\xi^p f)_{\l_i+ \l_j ,\l_1,\cdots, \hat{\l_{i}},\cdots, \hat{\l_{j}},\cdots, \l_{p}}([{P(x_i)}_{\l_i}x_j], x_1, \cdots, \hat{x_i},\cdots, \hat{x_j},\cdots,x_{p+1})\\&=\sum_{i=1}^{p+1}(-1)^{i+1} \rho(P(x_i))_{\l_i} f_{\l_1,\l_2,\cdots,\hat{\l_i},\cdots,\l_{p}}(Px_1,Px_2,\cdots,\hat{Px_i},\cdots,Px_{p+1})\\&
 - \sum_{i=1}^{p+1}(-1)^{i+1} \rho(P(x_i))_{\l_i}(\phi \circ f)_{\l_1,\l_2,\cdots,\hat{\l_i},\cdots,\l_{p}}(Px_1,x_2,\cdots,\hat{x_i},\cdots,x_{p+1})
 \\& + \sum_{1\leq i<j\leq p+1}(-1)^{i+j} f_{\l_i+ \l_j ,\l_1,\cdots, \hat{\l_{i}},\cdots, \hat{\l_{j}},\cdots, \l_{p}}(P[{P(x_i)}_{\l_i}x_j] ,Px_1,Px_2,\cdots\hat{Px_i},\cdots\hat{Px_j},\cdots,Px_p)\\&- \sum_{1\leq i<j\leq p+1}(-1)^{i+j}(\phi \circ f)_{\l_i+\l_j,\l_1,\l_2,\cdots,\hat{\l_i},\cdots\hat{\l_j},\cdots,\l_{p}}(P[{P(x_i)}_{\l_i}x_j],x_1,x_2,\cdots\hat{x_i},\cdots\hat{x_j},\cdots,x_p).
\end{align*} 
By using Eq.\eqref{eqrep}, we see that \begin{align}
        \xi^{p+1}(\delta^p(f))_{\l_1,\l_2,\cdots,\l_{p}}(x_1,x_2,\cdots,x_{p+1}) = \partial^p
(\xi^{p}(f))_{\l_1,\l_2,\cdots,\l_{p}}(x_1,x_2,\cdots,x_{p+1}).
   \end{align} By now we complete the proof.
\end{proof} The Lemma \ref{lemma3.2} yields the following commutative diagram,
$$\begin{tikzcd}
\cdots \arrow{r} & C^{p-1}(\mathfrak{L}, \mathfrak{M}) \arrow{r}{\delta^{p-1}} \arrow{d}[swap]{\xi^{p-1}} & C^p(\mathfrak{L}, \mathfrak{M}) \arrow{r}{\delta^p} \arrow{d}[swap]{\xi^p} & C^{p+1}(\mathfrak{L}, \mathfrak{M}) \arrow{r}{\delta^{p+1}} \arrow{d}[swap]{\xi^{p+1}} & C^{p+2}(\mathfrak{L}, \mathfrak{M}) \arrow{r} \arrow{d}{\xi^{p+2}} & \cdots \\
\cdots \arrow{r} & C^{p-1}_{AO}(\mathfrak{L}, \mathfrak{M}) \arrow{r}{\partial^{p-1}} & C^p_{AO}(\mathfrak{L}, \mathfrak{M}) \arrow{r}{\partial^p} & C^{p+1}_{AO}(\mathfrak{L}, \mathfrak{M}) \arrow{r}{\partial^{p+1}} & C^{p+2}_{AO}(\mathfrak{L}, \mathfrak{M}) \arrow{r} & \cdots.
\end{tikzcd}$$
\subsection{Cohomology complex of averaging Lie conformal algebra}
   After presenting the cohomology of Lie conformal algebra and averaging operators, our next task is to provide the cohomology of averaging Lie conformal algebra $\mathfrak{L}_P$ with representation $\mathfrak{M}_\phi$. To acheive this, we combine the cochain complex of Lie conformal algebra and the cochain complex of averaging operator to define the cochain complex of averaging Lie conformal algebra. Now, we define the cochain groups by $C^0_{AL}(\mathfrak{L}, \mathfrak{M}) = C^0(\mathfrak{L}, \mathfrak{M})$ and $C^p_{AL}(\mathfrak{L}, \mathfrak{M}) = C^p(\mathfrak{L},\mathfrak{M}) \oplus C^{p-1}_{AO}(\mathfrak{L},\mathfrak{M})$, $ \forall~ p \geq 1$, and the coboundary map $d^p_{AL}: C^p_{AL}(\mathfrak{L},\mathfrak{M}) \to C^{p+1}_{AL}(\mathfrak{L},\mathfrak{M})$ given by
\begin{align*}
    d^p_{AL}(f,g) = (\delta^p(f),- \xi^p(f)-\partial^{p-1}(g)),
\end{align*}
for any $f \in C^p(\mathfrak{L},\mathfrak{M})$ and $g \in C^{p-1}_{AO}(\mathfrak{L},\mathfrak{M})$.
\begin{thm}
The map $d^p_{AL} : C^p_{AL}(\mathfrak{L},\mathfrak{M}) \to C^{p+1}_{AL}(\mathfrak{L},\mathfrak{M})$ satisfies $d^{p+1}_{AL} \circ d^p_{AL} = 0$.
\end{thm}
\begin{proof}
    Let $f \in C^p(\mathfrak{L},\mathfrak{M})$ and $g \in C^{p-1}_{AO}(\mathfrak{L},\mathfrak{M})$, then we have
\begin{align*}
    d^{p+1}_{AL} \circ d^p_{AL}(f,g) &= d^{p+1}_{AL}(\delta^p(f),-\partial^{p-1}(g) - \xi^p(f)) \\&= (\delta^{p+1}(\delta^p(f)),-\partial^p(-\partial^{p-1}(g) - \xi^p(f)) - \xi^{p+1}(\delta^p(f)))\\&= (0, \partial^p(\xi^p(f)) - \xi^{p+1}(\delta^p(f))) \\&= 0.
\end{align*}\end{proof}
Thus, $\{C^p_{AL}(\mathfrak{L}_P,\mathfrak{M}_\phi),d^p_{AL}\}$ is the cochain complex of averaging Lie conformal algebras. Its cohomology class is denoted by $H^*_{AL}(\mathfrak{L}_P,\mathfrak{M}_\phi)$.
All the cochain complexes we studied above can be expressed in the form of the following short exact sequence
\begin{align}
    0 \longrightarrow C^p_{\text{AO}}(\mathfrak{L},\mathfrak{M}) \longrightarrow C^p_{\text{AL}}(\mathfrak{L},\mathfrak{M}) \longrightarrow C^p (\mathfrak{L},\mathfrak{M}) \longrightarrow 0.
\end{align}

   \section{$2$-terms homotopy averaging Lie conformal algebra}
  The homotopy version of conformal algebras or the conformal version of homotopy algebra is a recent topic of study. Lately, only a few researches have been carried out in this direction. For instance, the homotopy theory of associative conformal algebras is explored in \cite{SD}. This concept for $\mathfrak{L}_{\infty}$ algebras is introduced in \cite{KS} and the homotopy theory of Leibniz conformal algebra has been studied in \cite{DS}. Motivated by these studies, this section is devoted to introducing homotopy averaging operators on $2$-term $\mathfrak{L}_\infty$ conformal algebras. A $2$-term $\mathfrak{L}_\infty$-conformal algebra equipped with a homotopy averaging operator is called a $2$-term averaging $\mathfrak{L}_\infty$-conformal algebra. We focus on ``skeletal" and ``strict" $2$-term averaging $\mathfrak{L}_\infty$-conformal algebras. In particular, we show that skeletal $2$-term averaging $\mathfrak{L}_\infty$-conformal algebras correspond to $3$-cocycles of averaging Lie conformal algebras. Next, we introduce crossed modules of averaging Lie conformal algebras and show that crossed modules of averaging Lie conformal algebras correspond to strict $2$-term averaging $\mathfrak{L}_\infty$-conformal algebras.
\begin{defn}
   An $\mathfrak{L}_\infty$-conformal algebras are graded $\mathbb{C}[\p]$-modules $\mathfrak{L}=\oplus_i\mathfrak{L}_i$ where $i\in  \mathbb{Z}$, equipped with the collection of graded $C$-linear maps $\{l_k:\mathfrak{L}^{\otimes k}\to \mathfrak{L}[\l_1,\l_2,\cdots, \l_{k-1}]\}$ with degree $ k-2$ for any $k\in \mathbb{N}$, satisfying the following set of equations in the sense that
   \begin{enumerate}
       \item $l_k$ is \textbf{conformal sesquilinear}, i.e.,
       \begin{equation}\label{eq10}
	\begin{aligned} {l_k}_{\l_1, \l_2,\cdots, \l_{k-1}}&(x_1,x_2,\cdots,\p(x_i),\cdots,x_k)\\&=\begin{cases}
	-\l_{i} {l_{k}}_{\l_1,\l_2,\cdots,\l_{k-1}}(x_1, x_2, \cdots , x_k),&i= 1, \cdots, k-1, \\
	 (\p+\l_1+\l_2+\cdots+\l_{k-1}) {l_k}_{\l_1,\l_2,\cdots,\l_{k-1}}(x_1, x_2, \cdots , x_k),&i=k. \end{cases}
	\end{aligned}\end{equation}
       \item $l_k$ is \textbf{conformal skew-symmetric}, i.e.,
       \begin{align*}
{l_k}_{\l_1+\l_2+\cdots+\l_{k-1}}(x_1,\cdots x_i, x_{i+1}, \cdots, x_k)= sgn(\sigma) \epsilon(\sigma){l_k}_{\l_{\sigma(1)},\cdots,\l_{\sigma(k-1)}}(x_{\sigma(1)},\cdots,x_{\sigma(i+1)},x_{\sigma(i)},\cdots,x_{\sigma(k)})|_{\l_k \mapsto \l_k^\dagger}\end{align*}
	where $\l_{k}^{\dagger}= -\sum_{i=1}^{k-1}\l_{i}-\p$, for all $\sigma\in \mathbb{S}_k$.
 \item \textbf{Higher conformal Jacobi identity}, i.e., for any $k\in N$ and homogeneous element $x_1, x_2, \cdots x_n \in \mathfrak{L}$
 \begin{align*}
     \sum_{i+j=k+1}\sum_{\sigma\in S_k}sgn(\sigma)\epsilon(\sigma)(-1)^{j(i-1)}{l_p}_{\l_{\sigma(1)}+\l_{\sigma(2)}+\cdots+\l_{\sigma(q)},\l_{\sigma(q+1)},\cdots,\l_{\sigma(k)}}({l_q}_{\l_{\sigma(1)},\cdots,\l_{\sigma(q-1)}}&(x_{\sigma(1)},\cdots, x_{\sigma(q)}),\\& x_{\sigma(q+1)}, \cdots, x_{\sigma(k)}).
 \end{align*}Where $\sigma\in Sh(q,n-q)$ means that either $\sigma(q)=n$ or $\sigma(n)=n$. It have effect on $\l_{\sigma}^{\dagger}$, i.e., if $\sigma(q)=n$, then we use the notation $\l_{\sigma}^{\dagger}=\l_{\sigma(1)}+\cdots+\l_{\sigma(q)}$ and when  $\sigma(n)=n$, then we use the notation $\l_{\sigma}^{\dagger}=-\p-\l_{\sigma(q+1)}-\cdots-\l_{\sigma(k)}.$ 
   \end{enumerate}
\end{defn} Note that, Lie conformal algebra and differential graded Lie conformal algebra are the particular case of the $\mathfrak{L}_\infty$-conformal algebra. However, we are also familiar with the $2$-term $\mathfrak{L}_{\infty}$-algebra, defined in \cite{ADSS} a particular case of $\mathfrak{L}_\infty$-conformal algebra. In the following definition, we define $2$-term $\mathfrak{L}_\infty$-conformal algebra: 
\begin{defn}
 A $2$-term $\mathfrak{L}_\infty$-conformal algebra is a triple consisting of 
   \begin{enumerate}\item[(i)] a 
 chain complex of $\mathbb{C}[\p]$-modules $d:\mathfrak{L}_1\to \mathfrak{L}_0$,
   \item[(ii)] a conformal sesqui-linear and skew-symmetric $\mathbb{C}$-bilinear map $[[\cdot_\l \cdot]]: \mathfrak{L}_i\otimes \mathfrak{L}_j\to \mathfrak{L}_{i+j}[[\l]]$, for $i,j,i+j\in [0, 1]$, and
     \item[(iii)] a conformal sesqui-linear skew-symmetric trilinear map $l_3:\mathfrak{L}_0\otimes \mathfrak{L}_0\otimes \mathfrak{L}_0\to \mathfrak{L}_{1}[[\l,\m]]$,
   \end{enumerate}
that satisfy the following set of identities :
\begin{enumerate}
    \item[(L1)] $[[m_\l n]]=0$,
    \item[(L2)] $[[x_\l m]]=-[[m_{-\p-\l}x]]$,
   \item[(L3)] $[[x_\l y]]=-[[y_{-\p-\l}x]]$,
\item [(L4)]$d ([[x_\l m]])=  [[x_\l dm]]$   ,
\item [(L5)]$ [[dm_\l n]]= [[m_\l dn]]$  ,
    \item [(L6)] $d((l_3)_{\l,\m}(x,y,z))= ([[x_\l ([[y_\m z]])]]-[[([[x_\l y]])_{\l+\m},z]]- [[y_\m ([[x_\l z]])]]$,
    \item [(L7)]$(l_3)_{\l,\m}(x,y,d m)= [[x_\l ([[y_\m m]])]] - [[([[x_{\l}y]])_{\l+\m}m]] - [[y_\m ([[x_\l m]])]]$,
    \item[(L8)] \begin{align*}
       & [[x_\l {l_3}_{\m,\nu}(y,z,w)]] - [[y_\m {l_3}_{\l,\nu}(x,z, w)]] +  [[z_{\nu} {l_3}_{\l,\m}(x,y, w)]]- [[w_{-\p-\l-\m-\nu} {l_3}(x,y, z)]] \\&= {l_3}_{\l+\m,\nu}([[x_\l y]] ,z,w) +{l_3}_{\m,\l+\nu}(y, [[x_\l z]] ,w)+ {l_3}_{\m,\nu}(y,z, [[x_\l w]]) \\& +{l_3}_{\l,\m+\nu}(x, [[y_\m z]], w) - {l_3}_{\l,\nu}(x,z, [[y_\m w]])+ {l_3}_{\l,\m}(x,y ,[[z_\nu w]]),
    \end{align*} for all $x,y,z,w\in \mathfrak{L}_0$ and $m,n\in \mathfrak{L}_1$.\end{enumerate}\end{defn} It follows from the above definition that $2$-term $\mathfrak{L}_\infty$-conformal algebra is nothing but an $\mathfrak{L}_\infty$-conformal algebra, whose underlying graded $\mathbb{C}[\p]-$module is concentrated in degree $0$ and $1$, i.e., $\mathfrak{L}=\mathfrak{L}_0+\mathfrak{L}_1$. This implies that $l_k=0$ for $k\geq 4$.\\
    In the following, we define the conformal morphisms of the two $2$-term $\mathfrak{L}_\infty$-conformal algebras.
    \begin{defn}
        Let $\mathfrak{L} =(d:\mathfrak{L}_1\to \mathfrak{L}_0,[[\cdot_\l \cdot]], l_3)$ and $\mathfrak{L}' =(d':\mathfrak{L}_1'\to \mathfrak{L}'_0,[[\cdot_\l \cdot]]', l'_3)$ are two $2$-term $\mathfrak{L}_\infty$-conformal algebras. A homomorphism of these conformal algebras is given by a triple $(f_0,f_1,f_2)$, where $f_0: \mathfrak{L}_0\to \mathfrak{L}'_0$, $f_1: \mathfrak{L}_1\to \mathfrak{L}'_1$ are  $\mathbb{C}[\p]$-linear maps,  and  $f_2: \mathfrak{L}_0\otimes \mathfrak{L}_0\to \mathfrak{L}_1'[\l]$ is a conformal sesqui-linear skew-symmetric map such that following equations hold for all $x, y, z \in \mathfrak{L}_0$ and $m \in \mathfrak{L}_1$:
        \begin{enumerate}
        \item[(H1)] $f_{0} d = d' f_1,$ 
         \item[(H2)] $d'({f_2}_\l(x,y)) =  [[{f_0(x)}_\l f_0(y) ]] - f_0([[x_\l y]])$,
         \item[(H3)] ${f_2}_\l(x, dm) =- f_1([[x_\l m ]]) + [[f_0(x)_\l f_1(m)]],$  
      \item[(H4)] ${f_2}_\l(dm,x) =- f_1([[m_\l x ]]) + [[{f_1(m)}_\l f_0(x)]] ,$
            \item[(H5)]  \begin{align*}{l_3}_{\l,\m}(f_0(x), f_0(y), f_0(z)) - f_1 {l_3}_{\l,\m}(x, y, z)&= [[{f_0(x)}_\l {f_2}_\m(y, z)]] + {f_2}_\l  (x, ([[{y}_\m z]]))\\&  - [[{f_0(y)}_\m {{f_2}_\l(x, z)}]] -{f_2}_\m (y, [[{f_0(x)}_\l z]])\\& - [[{{f_2}_\l (x,y)}_{\l+\m } f_0 (z)]]- {f_2}_{\l+\m}([[{f_0(x)}_\l y]] , z). \end{align*}
        \end{enumerate}
    \end{defn}
    Further suppose that $\mathfrak{L} =(d:\mathfrak{L}_1\to \mathfrak{L}_0,[[\cdot_\l \cdot]], l_3)$  be $2$-term $\mathfrak{L}_\infty$-conformal algebra, then identity homomorphism $Id_\mathfrak{L}:\mathfrak{L}\to \mathfrak{L}$ is given by a triple $Id_\mathfrak{L}= (Id_{\mathfrak{L}_0}, Id_{\mathfrak{L}_1},0)$, It corresponds the following identities:
    \begin{itemize}
        \item $f_{0} d = d f_1,$
        \item $0 =  [[{f_0(x)}_\l f_0(y) ]] - f_0([[x_\l y]]),$
         \item $0 =- f_1([[x_\l m]]) + [[f_0(x)_\l f_1(m)]],$  
      \item $0 =- f_1([[m_\l x ]]) + [[{f_1(m)}_\l f_0(x)]], $
            \item  ${l_3}_{\l,\m}(f_0(x), f_0(y), f_0(z)) - f_1 {l_3}_{\l,\m}(x, y, z) = 0 .$
    \end{itemize}
\begin{defn}
    Let $\mathfrak{L} =(d:\mathfrak{L}_1\to \mathfrak{L}_0,[[\cdot_\l \cdot]], l_3)$ be a $2$-term $\mathfrak{L}_\infty$-conformal algebra. A homotopy averaging operator on $\mathfrak{L}$ is a triple $P =(P_0, P_1, P_2),$  where $P_0 : \mathfrak{L}_{0} \to \mathfrak{L}_0$ and  $P_1 : \mathfrak{L}_{1} \to \mathfrak{L}_1$ are linear maps and $P_2 : \mathfrak{L}_0\times \mathfrak{L}_0 \to \mathfrak{L}_1[\l]$ is a skew-symmetric bilinear map  such that for all $x, y, z \in \mathfrak{L}_0$ and $m\in \mathfrak{L}_1$,
  \begin{itemize}
      \item [(A1)]  $ P_{0} d = d P_1,$ 
      \item [(A2)] $d({P_2}_\l(x,y)) =  [[{P_0(x)}_\l P_0(y) ]] - P_0([[P_0(x)_\l y]])$, 
      \item [(A3)] ${P_2}_\l(x, dm) =- P_1([[P_0(x)_\l m ]]) + [[P_0(x)_\l P_1(m)]] = -P_1([[x_\l P_1(m) ]]) +[[{P_0}(x)_\l P_1(m)]] ,$  
      \item [(A4)] ${P_2}_\l(dm,x) =- P_1([[m_\l P_0(x) ]]) + [[{P_1(m)}_\l P_0(x)]]=- P_1([[{P_1(m)}_\l x ]]) + [[{P_1(m)}_\l P_0(x)]] ,$
     \item [(A5)] \begin{align*}&{l_3}_{\l,\m}(P_0(x), P_0(y), P_0(z)) - P_1 {l_3}_{\l,\m}(P_0(x), y, z)\\&= [[{P_0(x)}_\l {P_2}_\m(y, z)]] - [[{{P_2}_\l (x,y)}_{\l+\m } P_0 (z)]] - [[{P_0(y)}_\m {{P_2}_\l(x, z)}]] \\&+ {P_2}_\l  (x, ([[{P_0(y)}_\m z]])) - {P_2}_{\l+\m}([[{P_0(x)}_\l y]] , z) -{P_2}_\m (y, [[{P_0(x)}_\l z]]).\end{align*} 
  \end{itemize}
\end{defn}A $2$-term averaging $\mathfrak{L}_\infty$-conformal algebra is a $2$-term $\mathfrak{L}_\infty$-conformal algebra $\mathfrak{L} =(d:\mathfrak{L}_1\to \mathfrak{L}_0,[[\cdot_\l \cdot]], l_3)$ equipped with a homotopy averaging operator $P =(P_0, P_1, P_2)$ on it. We denote a $2$-term averaging $\mathfrak{L}_\infty$-conformal algebra as above by $\mathfrak{L}_P = (d: \mathfrak{L}_1\to \mathfrak{L}_0, [[\cdot_\l \cdot]], l_3, P_0, P_1, P_2)$ or simply by $\mathfrak{L}_P$.

\begin{defn} Let $\mathfrak{L}_P$ be a $2$-term averaging $\mathfrak{L}_\infty$-conformal algebra. It is said to be \textbf{skeletal} if $d =0$ and is said to be \textbf{strict} if $l_3 =0$ and $P_2 =0$.

If $\mathfrak{L}_P$ is skeletal, then 
\begin{enumerate}
    \item[(sk1)] $[[m_\l n]]=0$,
    \item[(sk2)] $[[x_\l m]]=-[[m_{-\p-\l}x]]$,
   \item[(sk3)] $[[x_\l y]]=-[[y_{-\p-\l}x]]$ ,  
    \item [(sk4)] $0= ([[x_\l ([[y_\m z]])]]-[[([[x_\l y]])_{\l+\m},z]]- [[y_\m ([[x_\l z]])]]$,
    \item [(sk5)]$0= [[x_\l ([[y_\m m]])]] - [[([[x_{\l}y]])_{\l+\m}m]] - [[y_\m ([[x_\l m]])]]$,
    \item[(sk6)] \begin{align*}
       & [[x_\l {l_3}_{\m,\nu}(y,z,w)]] - [[y_\m {l_3}_{\l,\nu}(x,z, w)]] +  [[z_{\nu} {l_3}_{\l,\m}(x,y, w)]]- [[w_{-\p-\l-\m-\nu} {l_3}(x,y, z)]] \\&= {l_3}_{\l+\m,\nu}([[x_\l y]] ,z,w) +{l_3}_{\m,\l+\nu}(y, [[x_\l z]] ,w)+ {l_3}_{\m,\nu}(y,z, [[x_\l w]]) \\& +{l_3}_{\l,\m+\nu}(x, [[y_\m z]], w) - {l_3}_{\l,\nu}(x,z, [[y_\m w]])+ {l_3}_{\l,\m}(x,y ,[[z_\nu w]]),
    \end{align*}
    \item [(sk7)] $0 =  [[{P_0(x)}_\l P_0(y) ]] - P_0([[P_0(x)_\l y]])$, 
      \item [(sk8)] $0 =- P_1([[P_0(x)_\l m ]]) + [[P_0(x)_\l P_1(m)]] = -P_1([[x_\l P_1(m) ]]) +[[{P_0}(x)_\l P_1(m)]] ,$  
      \item [(sk9)] $0 =- P_1([[m_\l P_0(x) ]]) + [[{P_1(m)}_\l P_0(x)]]=- P_1([[{P_1(m)}_\l x ]]) + [[{P_1(m)}_\l P_0(x)]] $,
     \item [(sk10)] \begin{align*}&{l_3}_{\l,\m}(P_0(x), P_0(y), P_0(z)) - P_1 {l_3}_{\l,\m}(P_0(x), y, z)\\&= [[{P_0(x)}_\l {P_2}_\m(y, z)]] - [[{{P_2}_\l (x,y)}_{\l+\m } P_0 (z)]] - [[{P_0(y)}_\m {{P_2}_\l(x, z)}]] \\&+ {P_2}_\l  (x, ([[{P_0(y)}_\m z]])) - {P_2}_{\l+\m}([[{P_0(x)}_\l y]] , z) -{P_2}_\m (y, [[{P_0(x)}_\l z]]).\end{align*} 
\end{enumerate}
If $\mathfrak{L}_P$ is {strict} if $l_3 =0$, then 
  \begin{enumerate}
    \item[(st1)] $[[m_\l n]]=0$,
    \item[(st2)] $[[x_\l m]]=-[[m_{-\p-\l}x]],$
   \item[(st3)] $[[x_\l y]]=-[[y_{-\p-\l}x]],$
\item [(st4)]$d ([[x_\l m]])=  [[x_\l dm]],$   
\item [(st5)]$ [[dm_\l n]]= [[m_\l dn]],$  
    \item [(st6)] $0= ([[x_\l ([[y_\m z]])]]-[[([[x_\l y]])_{\l+\m},z]]- [[y_\m ([[x_\l z]])]],$
    \item [(st7)]$ 0 = [[x_\l ([[y_\m m]])]] - [[([[x_{\l}y]])_{\l+\m}m]] - [[y_\m ([[x_\l m]])]],$
    \item [(st8)]  $ P_{0} d = d P_1,$ 
      \item [(st9)] $0 =  [[{P_0(x)}_\l P_0(y) ]] - P_0([[P_0(x)_\l y]])$, 
      \item [(st10)] $0=- P_1([[P_0(x)_\l m ]]) + [[P_0(x)_\l P_1(m)]] = -P_1([[x_\l P_1(m) ]]) +[[{P_0}(x)_\l P_1(m)]] ,$  
      \item [(st11)] $0 =- P_1([[m_\l P_0(x)]]) + [[{P_1(m)}_\l P_0(x)]]=- P_1([[{P_1(m)}_\l x ]]) + [[{P_1(m)}_\l P_0(x)]]. $\end{enumerate}
\end{defn}

The following result gives a characterization of skeletal $2$-term averaging $\mathfrak{L}_\infty$-conformal algebras in terms of $3$-cocycles of averaging Lie conformal algebras. 
From the definition of skeletal averaging Lie conformal algebra, it is clear that identities $sk3, sk4, sk7$ 
  on $\mathbb{C}[\p]$-module $\mathfrak{L}_0$ show that the $(\mathfrak{L}_0, [[\cdot_\l \cdot]], P_0)$ is a Lie conformal algebra denoted by ${\mathfrak{L}_0}_{P_0}$. On the other hand, $sk2, sk5, sk8, sk9$ on $\mathbb{C}[\p]$ module $\mathfrak{L}_1$ show that  ${\mathfrak{L}_1}_{P_1}=(\mathfrak{L}_1, [[\cdot_\l \cdot]], P_1)$ is conformal representation with the representation map $\rho_{\l}:\mathfrak{L}_0\times \mathfrak{L}_1\to \mathfrak{L}_1$, given by $\rho(x)_\l m= [[x_\l m]]$ for $x\in \mathfrak{L}_0$ and $m\in \mathfrak{L}_1$.
Furthermore, conditions $sk6$ and $sk10$ are equivalent to writing 
\begin{align*}
    &\delta(l_3) (w,x,y,z)=0,\\
     &\p^2(P_2)(x,y,z)={l_3}_{\l,\m}(P_0(x), P_0(y), P_0(z)) - P_1 {l_3}_{\l,\m}(P_0(x), y, z).
 \end{align*}  Therefore, $$d^3_{AL}(l_3, P_2) =( \delta^3(l_3), \p^2_{AO}(P_2) - l_3  P^{\otimes 3}_0 -P_1l_3(P_0 \otimes Id^{\otimes2}))=0.$$
Hence $(l_3, P_2)\in C^3_{AL}(\mathfrak{L}_P,\mathfrak{M}_\phi)$ is a $3$-cocycle.  The triple $(\mathfrak{L}_P, \mathfrak{M}_\phi, (l_3, P_2))$, we obtain has $1-1$ correspondence with the $2$-term averaging $\mathfrak{L}_\infty$-conformal algebra. The converse to this fact can easily be verified.
Thus, we have the following proposition:
\begin{prop}There is a $1 -1$ correspondence between skeletal $2$-term averaging $\mathfrak{L}_\infty$-conformal algebras $\mathfrak{L}_P$ and triples of the form $(\mathfrak{L}_P, \mathfrak{M}_\phi, (g, \tau) )$, where $\mathfrak{L}_P$ is an averaging Lie conformal algebra, $\mathfrak{M}_\phi$ is a representation and $(g, \tau)\in C^3_{AL}(\mathfrak{L}_P,\mathfrak{M}_\phi)$ is a $3$-cocycle.\end{prop}
The above result motivates us to consider the following notion. Let $\mathfrak{L}_P $ and $\mathfrak{L}'_{P'}$ are two skeletal $2$-term averaging $\mathfrak{L}_\infty$-conformal algebras on the same chain complex. They are said to be equivalent if $[[\cdot_\l \cdot]]=[[\cdot_\l \cdot]]' $, $P_0=P_0$, $P_1=P_1$ and there exists a conformal skew-symmetric bilinear map $f : \mathfrak{L}_0 \otimes \mathfrak{L}_0 \to \mathfrak{L}_1$ and a $\mathbb{C}$-linear map $\xi : \mathfrak{L}_0 \to \mathfrak{L}_1$ such that $(l'_3,P'_2)=(l_3,P_2)+d^2_{AL}((f,\xi))$, where $d^2_{AL}$ is the coboundary operator of the averaging Lie conformal algebra ${\mathfrak{L}_0}_{P_0}$ with coefficients in the representation ${\mathfrak{L}_1}_{P_1}$. Thus, we arrive at the following theorem:\begin{thm}There is a $1 -1$ correspondence between the equivalence class of skeletal $2$-term averaging $\mathfrak{L}_\infty$-conformal algebras $\mathfrak{L}_P$ and triples of the form $(\mathfrak{L}_P, \mathfrak{M}_\phi, (g,\tau))$, where $\mathfrak{L}_P $ is an averaging Lie conformal algebra, $\mathfrak{M}_\phi$ is a representation and $(g, \tau)\in H^3_{AL}(\mathfrak{L}_P,\mathfrak{M}_\phi)$ is a $3rd$-cohomology class.\end{thm}
Next, we introduce the crossed module of averaging Lie conformal algebra and characterize strict $2$-term $\mathfrak{L}_\infty$-conformal algebra. The concept of crossed module for an averaging Lie algebra was previously defined in \cite{DS}. Here, we extend this notion by introducing the crossed module for averaging Lie conformal algebras.
\begin{defn} A crossed module of averaging Lie conformal algebras is a quadruple $({\mathfrak{L}_1}_{P_1}, {\mathfrak{L}_0}_{P_0}, d, \rho)$, where ${\mathfrak{L}_1}_{P_1}$ and ${\mathfrak{L}_0}_{P_0}$ are both averaging Lie conformal algebras, $d:{\mathfrak{L}_1}_{P_1}\to {\mathfrak{L}_0}_{P_0}$ is an averaging Lie conformal algebra morphism and $\rho:\mathfrak{L}_0 \otimes \mathfrak{L}_1\to \mathfrak{L}_1$ is a conformal sesquilinear map that makes ${\mathfrak{L}_1}_{P_1}$ into a representation of the averaging Lie conformal algebra ${\mathfrak{L}_0}_{P_0}$ satisfying additional conditions:
     \begin{align} d(\rho(x)_\l m)=&[x_\l dm]_{\mathfrak{L}_0}\\ 
     \rho(d m)_\l n =&[m_\l n]_{\mathfrak{L}_1},
     \end{align} for all $x\in \mathfrak{L}_0$ and $m,n \in \mathfrak{L}_1$.
\end{defn}
\begin{prop}\label{directsum}
    Let $({\mathfrak{L}_1}_{P_1}, {\mathfrak{L}_0}_{P_0}, d, \rho)$ be a crossed module of averaging Lie conformal algebras. Then $(\mathfrak{L}_0 \oplus \mathfrak{L}_1, P_0 \oplus P_1)$ is an averaging Lie conformal algebra, where $\mathfrak{L}_0 \oplus \mathfrak{L}_1$ is equipped with the bracket
\begin{align}\label{crossedmod}
    [{(x, m)}_\l (y, n)] := ([x_\l y]_{\mathfrak{L}_0}, {\rho(x)}_\l n - {\rho(y)}_{-\p-\l} m + [m_\l n]_{\mathfrak{L}_1}),
\end{align}
for $(x,m), (y,n) \in \mathfrak{L}_0 \oplus \mathfrak{L}_1$.
\end{prop}
\begin{proof}
     Since $\mathfrak{L}_0, \mathfrak{L}_1$ are both Lie conformal algebras and $\rho : \mathfrak{L}_0 \otimes \mathfrak{L}_1\to \mathfrak{L}_1$ is a conformal sesquilinear map, it follows that $\mathfrak{L}_0 \oplus \mathfrak{L}_1$ is a Lie conformal algebra with the bracket \eqref{crossedmod}. Moreover, for any $(x, m), (y, n) \in \mathfrak{L}_0 \oplus \mathfrak{L}_1$, we have
\begin{align*}
[{(P_0 \oplus P_1)(x, m)}_\l (P_0 \oplus P_1)(y, n)] &= [{(P_0(x), P_1(m))}_\l (P_0(y), P_1(n))] \\
&= ([{P_0(x)}_\l P_0(y)]_{\mathfrak{L}_0}, {\rho (P_0(x))}_\l P_1(n) - {\rho(P_0(y))}_{-\p-\l} P_1(m) + [{P_1(m)}_\l P_1(n)]_{\mathfrak{L}_1}) \\
&= (P_0[P_0(x)_\l y]_{\mathfrak{L}_0}, P_1(\rho{(P_0(x))}_\l n) - P_1(\rho(y)_{-\p-\l} P_1(m)) + P_1[P_1(m)_\l n]_{\mathfrak{L}_1})\\
&= (P_0 \oplus P_1)[{(P_0(x), P_1(m))}_\l (y, n)] \\
&= (P_0 \oplus P_1)[{(P_0 \oplus P_1)(x, m)}_\l (y, n)].
\end{align*}
This shows that the map $P_0 \oplus P_1 : \mathfrak{L}_0 \oplus \mathfrak{L}_1 \to \mathfrak{L}_0 \oplus \mathfrak{L}_1$ is an averaging operator. This proves the result.
\end{proof}
\begin{thm}\label{thmcross}
There is a $1-1$ correspondence between strict $2$-term averaging $\mathfrak{L}_\infty$-conformal algebras and crossed modules of averaging Lie conformal algebras.
\end{thm}
\begin{proof}
Let $\mathfrak{L}_P = (\mathfrak{L}_1 \xrightarrow{d} \mathfrak{L}_0, [[\cdot_\l \cdot]], l_3 = 0, P_0, P_1, P_2 = 0)= (\mathfrak{L}_1 \xrightarrow{d} \mathfrak{L}_0, [[\cdot _\l \cdot]], P_0, P_1)$ be a strict $2$-term averaging $\mathfrak{L}_\infty$-conformal algebra. Then it follows from $st3$, $st6$ and $st9$ that $(\mathfrak{L}_0, [[\cdot _\l \cdot]])$ is a Lie conformal algebra and $P_0 : \mathfrak{L}_0 \to \mathfrak{L}_0$ is an averaging operator on it, i.e., ${\mathfrak{L}_0}_{P_0}$ is an averaging Lie conformal algebra. Next, we define a bilinear conformal bracket $[\cdot_\l \cdot]_{\mathfrak{L}_1} : \mathfrak{L}_1 \times \mathfrak{L}_1 \to \mathfrak{L}_1[\l]$  by
\begin{align*}
    [m_\l n]_{\mathfrak{L}_1} := [[dm_\l  n]], \quad \text{for } m,n \in \mathfrak{L}_1.
\end{align*}
From conditions $st2$, $st5$ and $st7$, we see that $(\mathfrak{L}_1, [\cdot_\l \cdot]_{\mathfrak{L}_1})$ is a Lie conformal algebra. Moreover, condition $st10$ yields $P_1 : \mathfrak{L}_1 \to \mathfrak{L}_1$, an averaging operator. Hence ${\mathfrak{L}_1}_{P_1}$ is also an averaging Lie conformal algebra. On the other hand, the conditions $st4$ and $st8$ imply that $d : {\mathfrak{L}_1}_{P_1} \to {\mathfrak{L}_0}_{P_0}$ is an averaging Lie conformal algebra morphism. Finally, we define a map $\rho : \mathfrak{L}_0 \times \mathfrak{L}_1 \to \mathfrak{L}_1$ given by
\begin{align}
    \rho(x)_\l m := [[x_\l m]], \quad \text{for } x \in \mathfrak{L}_0, m \in \mathfrak{L}_1.
\end{align}
It follows from $st7$ and $st10$ that $\rho$ makes ${\mathfrak{L}_1}_{P_1}$ into a representation of the averaging Lie conformal algebra ${\mathfrak{L}_0}_{P_0}$. We also have
\begin{align}
    d( \rho(x)_\l m) = d[[x_\l m]] = [[x_\l dm]] \quad \text{and} \quad {\rho(dm)}_\l n = [[dm_\l n]]  = [m_\l n]_{\mathfrak{L}_1},
\end{align}for $x \in \mathfrak{L}_0, m, n \in \mathfrak{L}_1$. Hence $( {\mathfrak{L}_0}_ {P_0},{\mathfrak{L}_1}_ {P_1}, d, \rho)$ is a crossed module of averaging Lie conformal algebras.

Conversely, let $({\mathfrak{L}_1}_ {P_1}, {\mathfrak{L}_0}_ {P_0}, d, \rho)$ be a crossed module of averaging Lie conformal algebras. Then it is easy to verify that $(\mathfrak{L}_1 \xrightarrow{d} \mathfrak{L}_0, [[\cdot_\l \cdot]], l_3 = 0, P_0, P_1, P_2 = 0)$ is a strict $2$-term averaging $\mathfrak{L}_\infty$-conformal algebra, where the bracket \([[\cdot_\l \cdot]] : \mathfrak{L}_i \times \mathfrak{L}_j \to \mathfrak{L}_{i+j}[\l]\) (for \(0 \leq i, j \leq 1\)) is given by
\begin{align*}
    [[x_\l y]] := [x_\l y]_{\mathfrak{L}_0}, \quad [[x_\l m]] = -[[m_{-\p-\l} x]] := \rho(x)_\l m \quad \text{and} \quad [[m_\l n]] := 0,
\end{align*}
for $x, y \in \mathfrak{L}_0, m, n \in \mathfrak{L}_1$. The above two correspondences are inverse to each other. This completes the proof.\end{proof}
Combining the Proposition \ref{directsum} and Theorem \ref{thmcross}, we get the following results.
\begin{prop}
    Let $\mathfrak{L}_P = (\mathfrak{L}_1 \xrightarrow{d} \mathfrak{L}_0, [[\cdot_\l \cdot]], l_3 = 0, P_0, P_1, P_2 = 0)$ be a strict $2$-term averaging $\mathfrak{L}_\infty$-conformal algebra. Then ${\mathfrak{L}_0 \oplus \mathfrak{L}_1}_ {P_0 \oplus P_1}$ is an averaging Lie conformal algebra with the Lie conformal bracket given by
\begin{align}
    [{(x, m)}_\l (y, n)] := ([[x_\l y]],[[x_\l n]] - [[y_{-\p-\l} m]] + [[dm_\l n]]),
\end{align}
for \((x, m), (y, n) \in \mathfrak{L}_0 \oplus \mathfrak{L}_1\).
\end{prop}
\begin{ex}
    Let $\mathfrak{L}_P$ be an averaging Lie conformal algebra. Then $(\mathfrak{L}_P, \mathfrak{L}_P, \mathrm{Id}, \mathrm{ad})$ is a crossed module of averaging Lie conformal algebras, where $\mathrm{ad}$ denotes the adjoint representation. Therefore, it follows that,
$\left( \mathfrak{L} \xrightarrow{\mathrm{Id}} \mathfrak{L}, [\cdot_\l \cdot ]_\mathfrak{L}, l_3 = 0, P_0 = P, P_1 = P, P_2 = 0 \right)$
is a strict $2$-term averaging $\mathfrak{L}_\infty$-conformal algebra.
\end{ex}
\begin{ex}
    Let \(\mathfrak{L}_P, \mathfrak{H}_Q\) be two averaging Lie conformal algebras and $f : \mathfrak{L}_P \to \mathfrak{H}_Q$ be an averaging Lie algebra morphism. Then $(\ker f, \mathfrak{L}, i, \mathrm{ad})$ is a crossed module of averaging Lie algebras, where $i: \ker f \to \mathfrak{L}$ is the inclusion map.
\end{ex}
\section{Non-abelian extension of averaging Lie conformal algebra}In this section, we consider the non-abelian extensions of averaging Lie conformal algebra $\mathfrak{L}_P$ by another averaging  Lie conformal algebra $\mathfrak{H}_Q$. We define a second non-abelian cohomology group, and show that the non-abelian extensions can be classified by the second non-abelian cohomology group of averaging Lie conformal algebras.
\begin{defn}
     Let $\mathfrak{L}_P$ and $\mathfrak{H}_Q$ be two averaging Lie conformal algebras. A non-abelian extension of $\mathfrak{L}_P$ by  $\mathfrak{H}_Q$, is an averaging  Lie conformal algebra $\mathfrak{E}_R$ together with a short exact sequence
\begin{align}\label{nonab}
   0 \longrightarrow \mathfrak{H}_Q \xrightarrow{i} \mathfrak{E}_R \xrightarrow{p} \mathfrak{L}_P \longrightarrow 0.
\end{align} of averaging Lie conformal algebras.
Often we denote a non-abelian extension as above simply by $\mathfrak{E}_R$.
Let $\mathfrak{E}_R$ and $\mathfrak{E}'_R$ be two
extensions of $\mathfrak{L}_P$ by  $\mathfrak{H}_Q$. They are called equivalent if there exists a momorphism $\tau : \mathfrak{E}_R\to\mathfrak{E}'_R$ of averaging Lie
conformal algebras 
such that the following diagram commutes
$$\begin{tikzcd}
0 \arrow{r} & \mathfrak{H}_Q \arrow{r}{i} \arrow{d}[swap]{=} & \mathfrak{E}_R \arrow{r}{p} \arrow{d}[swap]{\tau} & \mathfrak{L}_P \arrow{r}\arrow{d}[swap]{=} & 0\\
0 \arrow{r} & \mathfrak{H}_Q \arrow{r}{i'} & \mathfrak{E}'_{R'} \arrow{r}{p'} & \mathfrak{L}_P \arrow{r} &0.
\end{tikzcd}$$ The set of all equivalence classes of non-abelian extensions of $\mathfrak{L}_P$ by $\mathfrak{H}_Q$ is denoted by $Ext_{nab}(\mathfrak{L}_P , \mathfrak{H}_Q)$.
\end{defn}
The non-abelian extension of averaging Lie conformal algebra defined here is also
a $\mathbb{C}[\p]$-split extension. In this regard, we can require averaging Lie conformal algebra $\mathfrak{L}_P$ to be
projective as $\mathbb{C}[\p]$-module. It is a generalization of $\mathbb{C}[\p]$-split abelian extension of averaging Lie conformal algebras.\\
Let $({\mathfrak{L}_1}_{P_1}, {\mathfrak{L}_0}_{P_0}, d, \rho)$ be a crossed module of averaging Lie conformal algebras. Then the exact sequence
$$0 \longrightarrow {\mathfrak{L}_1}_{P_1} \xrightarrow{i}{\mathfrak{L}_0}_{P_0}\oplus{\mathfrak{L}_1}_{P_1}\xrightarrow{p} {\mathfrak{L}_0}_{P_0} \longrightarrow 0.$$
is a non-abelian extension of ${\mathfrak{L}_0}_{P_0}$ by  ${\mathfrak{L}_1}_{P_1}$, where the averaging Lie conformal algebra structure on ${\mathfrak{L}_0\oplus\mathfrak{L}_1}_{P_0\oplus P_1}$
is given in Proposition \ref{directsum}. Thus, it follows from Theorem \ref{thmcross} that a strict $2$-term averaging $\mathfrak{L}_\infty$-conformal algebra gives rise to a non-abelian extension of averaging Lie conformal algebras.\\ Given a non-abelian extension ${\mathfrak{E}_R}$ of ${\mathfrak{L}_P}$ by ${\mathfrak{H}_Q}$, we can define the section map $s:{\mathfrak{L}_P} \to {\mathfrak{E}_R}$ that satisfy $p\circ s=id_{\mathfrak{L}_P}$. Now we define two $\mathbb{C}[\p]$-module maps $\chi_\l:\mathfrak{L}_P\times\mathfrak{L}_P\to\mathfrak{H}_Q[\l]$ and $\rho: \mathfrak{L}_P \times \mathfrak{H}_Q \to \mathfrak{H}_Q[\l]$  and $\Phi : \mathfrak{L}_P\to \mathfrak{H}_Q$  by
\begin{align}\label{maps}
\chi_\l(x, y) &:= [s(x)_\l s(y)]_\mathfrak{E}  -s [x_\l y]_\mathfrak{L}\\
\rho(x)_\l(m) &:= [s(x)_\l  m]_\mathfrak{E}\\
\Phi(x) &:= R(s(x)) - s(P(x)).
\end{align}
Since exact sequence \eqref{nonab} defines a non-abelian extension of the averaging Lie conformal algebra $\mathfrak{L}$ by another averaging Lie conformal algebra $\mathfrak{H}$
(by forgetting the average operators), it follows from \cite{F} that non-abelian $2$-cocycle of $\mathfrak{L}_P$ with values in $\mathfrak{H}_Q$ satisfy following equations \begin{align}\label{2cocycle1}(\rho(x)_\l {\rho(y)}_\m - {\rho(y)} _\m{\rho(x)}_\l) (h) - \rho([x_\l y])_{\l+\m}(h) &= [\chi_\l(x, y)_{\l+\m} h]_\mathfrak{H}\\
    \rho(x)_\l \chi_\m(y, z) + \rho(y)_\m \chi_\l(z, x) + \rho(z)_{-\p-\l-\m} \chi_\l(x, y)&\nonumber\\ \label{2cocycle2}- \chi_{\l+\m}([y_\m z]_\mathfrak{L}, x) - \chi_{\l+\m}([z_{-\p-\l} x]_\mathfrak{L}, y) - \chi_{\l+\m}([x_\l y]_\mathfrak{L}, z)  &= 0,
\end{align}for all $x, y, z \in \mathfrak{L}$ and $ h \in  \mathfrak{H}$. In terms of $\mathfrak{L}_P$ and $\mathfrak{H}_Q$, the above expression can be expressed in the following form.
\begin{lem} The maps $\chi_\l$, $\psi$, and $\Phi$ defined above satisfy the following compatible conditions: for all $x, y \in \mathfrak{L}$ and $h \in \mathfrak{H}$,
\begin{align}\label{E1} {\rho(P(x))}_\l Q(h) &= Q{(\rho(P(x))}_\l h + Q[\Phi(x)_\l h]_\mathfrak{H} - [\Phi(x)_\l Q(h)]_\mathfrak{H} \\&\nonumber= Q(\rho(x)_\l Q(h)) - [\Phi(x)_\l  Q(h)]_\mathfrak{H},
\end{align}
\begin{align}\label{E2}
\chi_\l(P(x), P(y)) - Q(\chi_\l(P(x), y)) - \Phi[P(x)_\l y]_{\mathfrak{L}} + \rho(P(x))_\l\Phi(y) - \rho(P(y))_{-\p-\l}\Phi(x) &+\\\nonumber Q(\rho(y)_{-\p-\l}\Phi(x)) + [\Phi(x)_\l \Phi(y)]_{\mathfrak{H}}=& 0.
\end{align}
\end{lem}
\begin{proof}Consider that
\begin{align*}
   \rho(P(x))_\l Q(h)  &- Q(\rho(P(x))_\l h) - Q[\Phi(x)_\l h]_\mathfrak{H} + [\Phi(x)_\l Q(h)]_\mathfrak{H}
    \\&= [sP(x)_\l Q(h)]_{\mathfrak{E}} - Q[sP(x)_\l h]_{\mathfrak{E}} - Q[Rs(x)_\l h]_{\mathfrak{E}} + Q[s P(x)_\l h]_{\mathfrak{E}} + [Rs(x)_\l Q(h)]_{\mathfrak{E}} - [sP(x)_\l Q(h)]_{\mathfrak{E}}\\&
= -Q[Rs(x)_\l h]_{\mathfrak{E}} + [Rs(x)_\l R(h)]_{\mathfrak{E}} \quad (\text{as } Q = R|_{\mathfrak{H}})\\&
= 0.\end{align*}
Since $R$ \text{ is an embedding tensor}.
We also have
\begin{align*}
 \rho(P(x))_\l Q(h)    &- Q(\rho(x)_\l Q(h)) + [{\Phi(x)}_\l Q(h)]_\mathfrak{H}
\\&= [sP(x)_\l Q(h)]_{\mathfrak{E}} - Q[s(x)_\l Q(h)]_{\mathfrak{E}} + [Rs(x)_\l Q(h)]_{\mathfrak{E}} - [s P(x)_\l Q(h)]_{\mathfrak{E}}
\\&= -Q[s(x)_\l Q(h)]_{\mathfrak{E}} + [Rs(x)_\l R(h)]_{\mathfrak{E}} \quad (\text{as } Q = R|_{\mathfrak{H}})
\\&= 0.
\end{align*}
This proves the identities in \eqref{E1}. To prove the identity \eqref{E2}, we observe that
\begin{align*}
   &\chi_\l(P(x), P(y)) - Q(\chi_\l(P(x), y)) - \Phi[P(x)_\l y]_{\mathfrak{L}} \\&+ \rho(P(x))_\l\Phi(y) - \rho(P(y))_{-\p-\l}\Phi(x) + Q(\rho(y)_{-\p-\l}\Phi(x)) + [\Phi(x)_\l \Phi(y)]_{\mathfrak{H}}
\\&= [sP(x)_\l sP(y)]_{\mathfrak{E}} - s[P(x)_\l P(y)]_{\mathfrak{L}} - Q([sP(x)_\l s(y)]_{\mathfrak{E}} - s[P(x)_\l y]_{\mathfrak{L}}) - Rs[P(x)_\l y]_{\mathfrak{L}} \\&+ sP[P(x)_\l y]_{\mathfrak{L}}
+ [sP(x)_\l Rs(y)]_{\mathfrak{E}} - [s P(x)_\l s P(y)]_{\mathfrak{E}} - [sP(y)_{-\p-\l} Rs(x)]_{\mathfrak{E}} + [sP(y)_{-\p-\l} sP(x)]_{\mathfrak{E}}\\& + Q[s(y)_{-\p-\l}Rs(x)]_{\mathfrak{E}}
- Q[s(y)_{-\p-\l} sP(x)]_{\mathfrak{E}} + [Rs(x)_{\l} Rs(y)]_{\mathfrak{E}} - [Rs(x)_\l sP(y)]_{\mathfrak{E}} - [sP(x)_\l  Rs(y)]_{\mathfrak{E}} \\&+ [s P(x)_\l s P(y)]_{\mathfrak{E}}
\\&= - s[P(x)_\l P(y)]_{\mathfrak{L}} + Q( s[P(x)_\l y]_{\mathfrak{L}}) - Rs[P(x)_\l y]_{\mathfrak{L}} + sP[P(x)_\l y]_{\mathfrak{L}} + Q[s(y)_{-\p-\l} Rs(x)]_{\mathfrak{E}} + [Rs(x)_{\l} Rs(y)]_{\mathfrak{E}}\\&=0 .
\end{align*}The above expression vanishes as both $P$ and $R$ are embedding tensors and $Q = R|_{\mathfrak{H}}$. This completes the proof.
\end{proof}
\begin{defn}
   A triple $(\chi_\l,\rho,\Phi)$ is called non-abelian $2$-cocycle of $\mathfrak{L}$ with coefficients from $\mathfrak{H}$ satisfying the Eqs. \eqref{2cocycle1}, \eqref{2cocycle2}, \eqref{E1} and \eqref{E2}.   
\end{defn}
For a split sequence in the category of $\mathbb{C}[\p]$-modules, the section is not unique in general. For
two different sections, we have the following result.\\
Let $s' : \mathfrak{L} \to \mathfrak{E}$ be any other section of \eqref{nonab}. Define the map $\tau : \mathfrak{L} \to \mathfrak{H}$ by
\begin{align*}
    \tau(x) := s(x) - s'(x) \quad \text{for all } x \in \mathfrak{L}.
\end{align*}
Let $\chi'_\l, \rho', \Phi'$ be the maps induced by the section $s'$, given by \eqref{maps}. For all $x, y \in \mathfrak{L} $ and $h \in \mathfrak{H}$, we have:
\begin{align}\label{equivalent1}
    \rho(x)_\l h - \rho'(x)_\l h = [s(x)_\l h]_{\mathfrak{E}} - [s'(x)_\l h]_{\mathfrak{E}} = [\tau(x)_\l h]_{\mathfrak{H}},
\end{align}
\begin{align}\label{equivalent2}
    \chi_\l(x, y) - \chi'_\l(x, y) &= [s(x)_\l s(y)]_{\mathfrak{E}} - s[x_\l y]_{\mathfrak{L}} - [s'(x)_\l s'(y)]_{\mathfrak{E}} + s'[x_\l y]_{\mathfrak{L}}\\&\nonumber
= [s'(x)_\l (s - s')(y)]_{\mathfrak{E}} - [s'(y)_{-\p-\l} (s - s')(x)]_{\mathfrak{E}} - (s - s')[x_\l y]_{\mathfrak{L}} + [(s - s')(x)_\l (s - s')(y)]_{\mathfrak{H}}
\\&\nonumber = \rho'(x)_\l \tau(y) - \rho'(y)_{-\p-\l} \tau(x) - \tau[x_\l y]_{\mathfrak{L}} + [\tau(x)_\l\tau(y)]_{\mathfrak{H}},
\end{align} and
\begin{align}\label{equivalent3}
    \Phi(x) - \Phi'(x) &= (Rs - sP)(x) - (R{s'} - s'P)(x)\\&\nonumber= R(s - s')(x) - (s - s')P(x)\\&\nonumber= R \tau(x) - \tau P(x).
\end{align}
The above discussion leads us to the following definition.
\begin{defn}
    Two non-abelian $2$-cocycles $(\chi_\l,\rho,\Phi)$ and $(\chi'_\l,\rho',\Phi')$ of $\mathfrak{L}_P$ with coefficients from $\mathfrak{H}_Q$ are said to be equivalent if there exists a map $\tau: \mathfrak{L}_P \to \mathfrak{H}_Q$, that satisfies Eqs. \eqref{equivalent1}, \eqref{equivalent2} and  \eqref{equivalent3}. The set of all equivalence classes of non-abelian extensions of $\mathfrak{L}_P$ by $\mathfrak{H}_Q$ are denoted by $Ext^2_{nab}(\mathfrak{L}_P,\mathfrak{H}_Q).$ 
\end{defn}
We have the following proposition.
\begin{prop}
    Let $\mathfrak{E}$ and $\mathfrak{E}'$ be two equivalent non-abelian extensions  of $\mathfrak{L}_P$ by  $\mathfrak{H}_Q$, with different sections $s$ and $s'$ respectively. Then the choice of different sections preserves the equivalence relation between two $2$-cocycles $(\chi_\l,\rho,\Phi)$ and $(\chi'_\l,\rho',\Phi')$.
\end{prop}
\begin{proof}
Let $\mathfrak{E}_R$ and $\mathfrak{E}'_{R'}$ be two equivalent non-abelian extensions of $\mathfrak{L}_P$ by $\mathfrak{H}_Q$. If $s: \mathfrak{L} \to \mathfrak{E}$ is a section of the map $p$, then it is easy to observe that the map $s':= \tau \circ s$ is a section of the map $p'$. Let $(\chi'_\l, \psi', \Phi')$ be the non-abelian $2$-cocycle corresponding to the non-abelian extension $\mathfrak{E}'_{R'}$ and the section $s'$. Then we have \begin{align*}\chi'_\l(x, y) =& [{s'(x)}_\l s'(y)]_{\mathfrak{E}'} - s'[x_\l y]_{\mathfrak{L}}\\&
= [\tau \circ s(x)_\l \tau \circ s(y)]_{e'} - (\tau \circ s)[x_\l y]_{\mathfrak{L}} \\&
= \tau\left([s(x)_\l s(y)]_{\mathfrak{E}} - s[x_\l y]_{\mathfrak{L}}\right) = \chi_\l(x, y) \quad (\text{as } \tau|_{\mathfrak{H}} = \text{Id}_{\mathfrak{H}}),
\end{align*}

\begin{align*}\rho'(x)_\l(h) &= [s'(x)_\l h]_{\mathfrak{E}} = [\tau \circ s(x)_\l h]_{\mathfrak{E}'}
= \tau\left([s(x)_\l h]_{\mathfrak{E}}\right) = \rho(x)_\l(h) \quad (\text{as } \tau|_{\mathfrak{H}} = \text{Id}_{\mathfrak{H}}),
\end{align*}
and
\begin{align*}\Phi'(x) &= R'{s'(x)} - s'P(x)\\&= R'(\tau \circ s(x)) - (\tau \circ s)P(x)
\\&= \tau\left(Rs(x) - sP(x)\right) = \Phi(x) \quad (\text{as } \tau|_{\mathfrak{H}} = \text{Id}_{\mathfrak{H}}).
\end{align*}
Thus, we have $(\chi_\l, \rho, \Phi) = (\chi'_\l, \rho', \Phi')$. Hence, they give rise to the same element in $H^2_{nab}(\mathfrak{L}_P,\mathfrak{H}_Q)$. Therefore, there is a well-defined map $\Upsilon : \text{Ext}_{nab}(\mathfrak{L}_P, \mathfrak{H}_Q) \to H^2_{nab}(\mathfrak{L}_P, \mathfrak{H}_Q)$.

Conversely, let $(\chi_\l, \rho, \Phi)$  be a non-abelian $2$-cocycle of $\mathfrak{L}_P$ with values in $\mathfrak{H}_Q$. Define $\mathfrak{E} := \mathfrak{L} \oplus \mathfrak{H}$ with the skew-symmetric conformal bilinear bracket
\begin{align} [(x, h)_\l (y, k)]_{\mathfrak{E}} := ([x_\l y]_{\mathfrak{L}}, \rho(x)_\l k - \rho(y)_{-\p-\l} h + \chi_\l(x, y) + [h_\l k]_{\mathfrak{H}})
\end{align} for $(x, h), (y, k) \in \mathfrak{E} $. As noticed in \cite{F}, using the conditions (8) and (9), the bracket $[\cdot_\l \cdot ]_{\mathfrak{E}}$ satisfies the conformal Jacobi identity. In other words, $(\mathfrak{E}, [\cdot_\l \cdot]_{\mathfrak{E}})$ is a Lie conformal algebra. 

Further, we define a map $ R : \mathfrak{E} \to \mathfrak{E}$ by
\begin{align*}
    R(x, h) := (P(x), Q(h) + \Phi(x)),
\end{align*} for all $(x, h) \in \mathfrak{E}$. Then 
by using Eq. \eqref{E1}, we see that $R$ is an averaging operator on $\mathfrak{E}$  and corresponding averaging Lie conformal algebra is denoted by $\mathfrak{E}_R$.
Moreover, it is easy to see that
\begin{align}
    0 \longrightarrow \mathfrak{H}_Q \xrightarrow{i} \mathfrak{E}_R \xrightarrow{p} \mathfrak{L}_P \longrightarrow 0
\end{align}
is a non-abelian extension of the averaging Lie algebra $\mathfrak{L}_P$ by  $\mathfrak{H}_Q$, where $i(h) = (0, h)$ and $p(x, h) = x$ for all $(x, h) \in \mathfrak{E}$ and $h \in \mathfrak{H}$.

Next, let $(\chi_\l, \rho, \Phi)$ and $(\chi'_\l, \rho', \Phi') $ be two equivalent non-abelian $2$-cocycles, i.e., there exists a linear map $\tau : \mathfrak{L} \to \mathfrak{H}$ such that the identities \eqref{equivalent1}, \eqref{equivalent2}, and \eqref{equivalent3} hold. Let $\mathfrak{E}'_{R'}$ be the averaging Lie conformal algebra induced by the $2$-cocycle $(\chi', \rho', \Phi')$. Note that the Lie conformal algebra $\mathfrak{E}'=\mathfrak{L} \oplus \mathfrak{H} $ is a $\mathbb{C}[\p]$-module equipped with the Lie conformal bracket given by
\begin{align}
    [(x, h)_\l (y, k)]_{\mathfrak{E}'} := ([x_\l y]_{\mathfrak{L}}, \psi'(x)_\l k - \psi'(y)_{-\p-\l} h + \chi'_\l(x, y) + [h_\l k]_{\mathfrak{H}}),
\end{align} for $(x, h), (y, k) \in \mathfrak{E}' $. Moreover, the map $R' : \mathfrak{E}' \to \mathfrak{E}'$ is given by
$R'(x, h) := (P(x), Q(h) + \Phi'(x))$ for all $(x, h) \in \mathfrak{E}'$. 
We now define a map $\tau : \mathfrak{L} \oplus \mathfrak{H} \to \mathfrak{L} \oplus \mathfrak{H}$ by
$\phi(x, h) := (x, h + \tau(x))$
for all $(x, h) \in \mathfrak{L} \oplus \mathfrak{H}$. Then by a straightforward calculation, we have
$\phi([(x, h)_\l (y, k)]_{\mathfrak{E}}) = [\phi(x, h)_\l \phi(y, k)]_{\mathfrak{E}'}.$
Further,
\begin{align*} (U \circ \phi)(x, h) &= U'(x, h + \tau(x))
\\&= (P(x), Q(h) + Q(\tau(x)) + \Phi'(x))
\\&= (P(x), Q(h) + \tau(P(x)) + \Phi(x)) \quad (\text{by \eqref{equivalent3}})
\\&= \phi(P(x), Q(h) + \Phi(x)) = (\phi \circ U)(x, h).
\end{align*}
Hence the map $\phi: \mathfrak{E}_R \to \mathfrak{E}'_{R'}$ defines an equivalence between the two non-abelian extensions. Therefore, we obtain a well-defined map $\Lambda : H^2_{\text{nab}}(\mathfrak{L}_P, \mathfrak{H}_Q) \to \text{Ext}_{\text{nab}}(\mathfrak{L}_P, \mathfrak{H}_Q) $. Finally, it is straightforward to verify that the maps $\Upsilon$ and $\Lambda$ are inverses of each other. This completes the proof.\end{proof}
\section{Automorphisms of averaging Lie conformal algebra and Wells map}
In this section, we study the inducibility of a pair of automorphisms in the context of a non-abelian extension of averaging Lie conformal algebras and present the fundamental sequence of Wells. For this, we first define the automorphisms of averaging Lie conformal algebras, then establish the necessary and sufficient condition for these automorphisms to be inducible.

Let $\mathfrak{L}_P$ and $\mathfrak{H}_Q$ be two averaging Lie conformal algebras and \begin{align*}
    0 \longrightarrow \mathfrak{H}_Q \xrightarrow{i} \mathfrak{E}_R \xrightarrow{p} \mathfrak{L}_P \longrightarrow 0
\end{align*} be a non-abelian extension of $\mathfrak{L}_P$ by $\mathfrak{H}_Q$ with the section $s$. The corresponding non-abelian $2$-cocycles of $\mathfrak{E}_R$ is denoted by $(\chi_\l,\rho,\Phi)$. Let $Aut_\mathfrak{H}(\mathfrak{E}_R )$ be the set of all averaging Lie conformal algebra automorphisms $\gamma\in  Aut(\mathfrak{E}_R)$ that satisfies $\gamma|\mathfrak{H}\subset \mathfrak{H}$. More explicitly,
$Aut_\mathfrak{H}(\mathfrak{E}_R ) := \{\gamma \in Aut(\mathfrak{E}_R)| \gamma(\mathfrak{H}) = \mathfrak{H}\}.$

For any section $s : \mathfrak{L}\to\mathfrak{E}$ of the map $p$, for any $\gamma \in Aut_\mathfrak{H}(\mathfrak{E}_R)$, we can define a $\mathbb{C}[\p]$-module map $\bar{\gamma}: \mathfrak{L} \to \mathfrak{L}$ by $\bar{\gamma}(x) := p\gamma s(x)$, for $x \in \mathfrak{L}$. It is easy to verify that the map $\bar{\gamma}$ is independent of the choice of the section $s$. Moreover, $p$ is a projection on $\mathfrak{L}$ and $\gamma $ preserves $\mathfrak{L} $, so $\bar{\gamma}$ is a bijection on $\mathfrak{L}$.  For any $x, y \in \mathfrak{L} $, we have
\begin{align*}
 \bar{\gamma}([x_\l y]_\mathfrak{L}) = p {\gamma}(s[x, y]_\mathfrak{L}) &= p{\gamma}([s(x)_\l s(y)]_\mathfrak{E} -\chi_\l (x, y))
\\&= p{\gamma}[s(x)_\l s(y)]_\mathfrak{E} \quad(\textit{as $\gamma|\mathfrak{H} \subset \mathfrak{H}$ and $p|\mathfrak{H}= 0$})\\&= [p\gamma s(x)_\l p\gamma s(y)]_\mathfrak{L} = [\bar{\gamma}(x)_\l \bar{\gamma}(y)]_\mathfrak{L}.\end{align*}
and

\begin{align*}(P\bar{\gamma} - \bar{\gamma} P)(x) &= (P p\gamma s - p\gamma s P)(x) 
\\&= (pR \gamma s - p\gamma s P)(x)
\\&= p\gamma(Rs - sP)(x) = 0 \quad (\text{as } \gamma|_\mathfrak{H} \subset \mathfrak{H} \text{ and } p|_\mathfrak{H} = 0).
\end{align*}
This shows that the map $\bar{\gamma} : \mathfrak{L} \to \mathfrak{L}$ is an automorphism of the averaging Lie conformal algebra $\mathfrak{L}_P$. In other words, $\bar{\gamma}\in Aut(\mathfrak{L}_P)$. Observe that, $\overline{\gamma_1 \gamma_2} = \overline{\gamma_1} \ \overline{\gamma_2}.$
Now  we obtain a group homomorphism
$\Pi : Aut_\mathfrak{H}(\mathfrak{E}_R ) \to  Aut(\mathfrak{H}_Q) \times Aut(\mathfrak{L}_P))$ given by $\gamma \mapsto (\gamma|\mathfrak{H}, \bar{\gamma})$.
In general, we say any pair $(\a,\b)\in Aut(\mathfrak{H}_Q) \times Aut(\mathfrak{L}_P)$ is called inducible if $(\a,\b)$ lies in the image of $\Pi$, given above.\\
Given any pair $(\a,\b)\in Aut(\mathfrak{H}_Q) \times Aut(\mathfrak{L}_P)$ of automorphisms of averaging Lie conformal algebra, we define a new triple $(\chi_\l^{(\a,\b)}, \rho^{(\a,\b)}, \Phi^{(\a,\b)})$ of $\mathbb{C}[\p]$-module maps
$$\chi_\l^{(\a,\b)}: \mathfrak{L}\otimes \mathfrak{L} \to \mathfrak{H}[\l], \quad \rho^{(\a,\b)}: \mathfrak{L} \otimes \mathfrak{H} \to \mathfrak{H}[\l], \quad \Phi^{(\a,\b)}: \mathfrak{L} \to \mathfrak{H}$$
by \begin{align}\label{newcocycle}
    \chi^{(\a,\b)}_\l(x, y) := \a \circ \chi_\l (\b^{-1}(x), \b^{-1}(y)), \quad \rho^{(\a,\b)}(x)_\l h := \a(\rho({\b^{-1}(x)})_{\l} \a^{-1}(h)), \quad \Phi^{(\a,\b)}(x) := \a \Phi(\b^{-1}(x)),
\end{align}
for $x, y \in \mathfrak{L}$ and $h \in \mathfrak{H}$. Then we have the following result.
\begin{lem}The triple $(\chi_\l^{(\alpha,\beta)}, \rho^{(\alpha,\beta)}, \Phi^{(\alpha,\beta)})$ is a non-abelian $2$-cocycle on $\mathfrak{L}_P$ with values in $\mathfrak{H}_Q$.
\end{lem}
\begin{proof}
    The triple $(\chi_\l, \rho, \Phi)$ being a non-abelian $2$-cocycle implies that the identities \eqref{2cocycle1}, \eqref{2cocycle2}, \eqref{E1} and \eqref{E2} hold. In these identities, if we replace \(x\), \(y\), and \(h\) by \(\beta^{-1}(x)\), \(\beta^{-1}(y)\), and \(\alpha^{-1}(h)\), respectively, we simply get the non-abelian 2-cocycle conditions for the triple \((\chi_\l^{(\alpha,\beta)}, \rho^{(\alpha,\beta)}, \Phi^{(\alpha,\beta)})\). 
\end{proof}
\begin{defn}
A map $W : Aut(\mathfrak{H}_Q)\times  Aut(\mathfrak{L}_P ) \to H^2_{nab}(\mathfrak{L}_P , \mathfrak{H}_Q)$ by
$W((\a,\b)) = [(\chi_\l^{(\a,\b)}, \rho^{(\a,\b)}, \Phi^{(\a,\b)}) - (\chi_\l, \rho, \Phi)],$ the equivalence class of $(\chi_\l^{(\alpha,\beta)}, \rho^{(\alpha,\beta)}, \Phi^{(\alpha,\beta)})- (\chi_\l, \rho, \Phi).$ The map $W$ is called Wells map.\end{defn} 

Note that non-abelian $2$-cocycles depend on the choice of section, but Wells map does not depend on the choice of section. This can be understood by the fact that equivalent $2$-cocycles belong to the same cohomology class as described in the previous section. Since Wells map is related to these cohomology classes rather than the specific cocycles, that is why it is independent of the choice of section. Thus, different sections do not affect the outcome of the Wells map, as they yield equivalent cocycles in cohomology. This can be stated and explained in the form of a proposition given below:
\begin{prop} Wells map is independent of the choice of section.
\end{prop}
\begin{proof} Let $ s'$ be any other section of the map $p$, and let $(\chi_\l', \rho', \Phi')$ be the corresponding non-abelian $2$-cocycle. We have seen that the non-abelian $2$-cocycles $(\chi_\l, \rho, \Phi)$ and $(\chi_\l', \rho', \Phi')$ are equivalent by the map $\tau := s - s'$. Using this, it is easy to verify that the non-abelian $2$-cocycles $(\chi_\l^{(\alpha, \beta)}, \rho^{(\alpha, \beta)}, \Phi^{(\alpha, \beta)})$ and $(\chi_\l'^{(\alpha, \beta)}, \rho'^{(\alpha, \beta)}, \Phi'^{(\alpha, \beta)})$ are equivalent by the map $\alpha \tau \beta^{-1}$.Combining these results, we observe that the $2$-cocycles $(\chi_\l^{(\alpha, \beta)}, \rho^{(\alpha, \beta)}, \Phi^{(\alpha, \beta)}) - (\chi_\l, \rho, \Phi)$ and $(\chi_\l'^{(\alpha, \beta)}, \rho'^{(\alpha, \beta)}, \Phi'^{(\alpha, \beta)}) - (\chi_\l', \rho', \Phi')$ are equivalent by the map $\alpha \tau \beta^{-1} - \tau$. Therefore, their corresponding equivalence classes in $H^2_{\text{nab}}({\mathfrak{L}}_P, {\mathfrak{L}}_Q)$ are same. In other words, the map $W$ does not depend on the chosen section.
\end{proof}
\begin{thm}\label{thm6.4}
    Let $\mathfrak{L}_P$ and $\mathfrak{H}_Q$ be two averaging Lie conformal algebras and \begin{align}
    0 \longrightarrow \mathfrak{H}_Q \xrightarrow{i} \mathfrak{E}_R \xrightarrow{p} \mathfrak{L}_P \longrightarrow 0
\end{align} be a non-abelian extension of $\mathfrak{L}_P$ by $\mathfrak{H}_Q$ with the section $s$. The corresponding non-abelian $2$-cocycles of $\mathfrak{E}_R$ is denoted by $(\chi_\l,\rho,\Phi)$. A pair $(\alpha,\beta) \in  Aut(\mathfrak{H}_Q) \times Aut(\mathfrak{L}_P)$ is inducible if and
only if the non-abelian $2$-cocycles $(\chi_\l,\rho,\Phi)$ and $(\chi_\l^{(\alpha,\beta)}, \rho^{(\alpha,\beta)}, \Phi^{(\alpha,\beta)})$ are equivalent. In other words Wells map of a pair of automorphisms $(\a,\b)$ is zero.
\end{thm}
\begin{proof}For the proof of this theorem, we refer readers  to Theorem 5.5 of \cite{DS}.\end{proof}

\par In the Lie conformal algebra context,  we can obtain  Wells exact sequence of Lie conformal algebras. Let's generalize the exact sequence of Wells maps in the context of averaging Lie conformal algebras. For any non-abelian extension of averaging Lie conformal algebras $0 \longrightarrow \mathfrak{H}_Q \xrightarrow{i} \mathfrak{E}_R \xrightarrow{p} \mathfrak{L}_P \longrightarrow 0$. Define a subgroup $Aut^{\mathfrak{L},\mathfrak{H}}_{\mathfrak{H}}(\mathfrak{E}_R) \subset Aut_{\mathfrak{H}}(\mathfrak{E}_R)$ by
$Aut^{\mathfrak{L},\mathfrak{H}}_{\mathfrak{H}}(\mathfrak{E}_R):= \{\gamma\in  Aut_{\mathfrak{H}}(\mathfrak{E}_R) \mid \Pi(\gamma) = (Id_\mathfrak{H},Id_\mathfrak{L})\}.$

\begin{thm}\label{thm6.5} Let  $0 \longrightarrow \mathfrak{H}_Q \xrightarrow{i} \mathfrak{E}_R \xrightarrow{p} \mathfrak{L}_P \longrightarrow 0$ be a non-abelian extension of averaging Lie conformal algebras. Then there is an exact sequence
$1 \longrightarrow Aut^{\mathfrak{L},\mathfrak{H}}_{\mathfrak{H}}(\mathfrak{E}_R)\xrightarrow{i} Aut_{\mathfrak{H}}(\mathfrak{E}_R) \xrightarrow{\Pi} Aut(\mathfrak{H}_Q) \times Aut(\mathfrak{L}_P)\xrightarrow{W} H^2_{nab}(\mathfrak{L}_P,\mathfrak{H}_Q).$
\end{thm}
\begin{proof} The proof of this theorem follows the same principle as the Theorem 5.6 of \cite{DS}.
\end{proof}

\end{document}